\documentclass[11pt,a4paper]{amsart}
\usepackage[utf8]{inputenc}
\usepackage[english]{babel}
\usepackage{kpfonts}

\usepackage{tikz}
\usetikzlibrary{decorations.markings}

\usepackage{geometry}
\usepackage{amsmath}
\usepackage{graphicx, color}
\usepackage{amscd}
\usepackage{amsfonts}
\usepackage{amssymb}
\usepackage{mathrsfs}
\usepackage{mathtools}
\usepackage{ulem}
\usepackage{enumerate}
\usepackage{hyperref}
\usepackage{lipsum}

\newcommand{\Lop}{\mathcal{L}}
\newcommand{\R}{\mathbb{R}}
\newcommand{\PV}{\mathrm{P.V.}}

\newcommand{\inSet}[1]{\text{ in } {#1}}
\newcommand{\N}{\mathbb{N}}
\newcommand{\Ns}{\mathcal{N}_s}
\newcommand{\Cub}{C_{\mathrm{ub}}}
\newcommand{\Cloc}{C_{\mathrm{loc}}}
\newcommand{\Lomega}{L^1_{\omega_s}(\R)}

\newcommand{\dx}{\mathrm{d}x}
\newcommand{\dy}{\mathrm{d}y}
\newcommand{\dz}{\mathrm{d}z}
\newcommand{\dt}{\mathrm{d}t}
\newcommand{\dr}{\mathrm{d}r}
\newcommand{\ds}{\mathrm{d}s}

\newcommand{\sgn}{\mathrm{sgn}}
\renewcommand{\div}{\mathrm{div}}
\renewcommand{\phi}{\varphi}
\renewcommand{\epsilon}{\varepsilon}

\newtheorem{mainthm}{Theorem}

\newtheorem{thm}{Theorem}[section]
\newtheorem{cor}[thm]{Corollary}
\newtheorem{lem}[thm]{Lemma}
\newtheorem{prop}[thm]{Proposition}
\newtheorem{defn}{Definition}
\newtheorem{rmk}{Remark}

\newcounter{fhyp}
\title{The moving patch model with fractional diffusion}
\author{Sebastián Flores-Sepúlveda$^{1}$}
\author{Gabrielle Nornberg$^{1,2}$}
\author{Alexander Quaas$^3$}

\address{$^{1}$ Centro de Modelamiento Matemático (CNRS IRL2807), Universidad de Chile, Santiago, Chile}
\address{$^{2}$ Departamento de Ingeniería Matemática, Universidad de Chile, Santiago, Chile}
\address{$^{3}$ Departamento de  Matem\'atica,  Universidad T\'ecnica Federico Santa Mar\'{i}a, Valpara\'{\i}so, Chile.}

\begin{document}

\begin{abstract}
In this paper we study the following one-dimensional reaction-diffusion problem
$$
u_t+(-\Delta)^s u=f(x-c t, u) \;\:\textrm{ in } \mathbb{R}\times (0,+\infty),
$$
where $s>\frac{1}{2}$, $c \in \mathbb{R}$ is a prescribed velocity, and $f$ is of KPP type, which describes the evolution of a population in an advective environment subjected to nonlocal diffusion. We suppose the environment is such that it is only advantageous in a bounded ``patch", outside of which the species dies at an asymptotically constant rate. 

We first derive an optimal solvability criteria for the corresponding traveling waves problem 
$$\Delta^s u+c u^{\prime}+f(x, u)=0 \;\:\textrm{ in } \mathbb{R},$$ through the first eigenvalue of the associated linearized elliptic operator with drift. Then we use this criteria to establish the long time  behavior of the solution to the parabolic problem, for any continuous bounded nonnegative initial data, leading the species either through their extinction or survival. 
Moreover, assuming that for $c=0$ the population survives, we show that there exist two positive critical speeds $c^{*}$ and $c^{**}$ such that for all $|c| <c^{*}$ the population persists, whereas and it perishes for $|c| >c^{**}$.



\end{abstract}

\maketitle

\section{Introduction}

We study the following one-dimensional reaction-diffusion equation
\[u_t + (-\Delta)^su = f(x-ct,u)\]
where $c \in \R$ is a prescribed velocity, and $f$ is of KPP type. This equation describes the evolution of a population in an advective environment. More precisely, we suppose the environment is such that it is only advantageous in a bounded ``patch", outside of which the species dies at an asymptotically constant rate. The questions we aim to answer in this paper is: Under what conditions on $c$ does the species persist? Under what conditions does it go extinct?

This problem was first studied, for the classical Laplacian, by Berestycki, Diekmann, Nagelkerke and Zegeling in \cite{B+09}. They obtain the existence of a treshold value $c^* >0 $ such that the species persists for all speeds $|c| < c^*$ and eventually goes extinct for all speeds $c\geq c^*$. This threshold value is closely related to the generalized principal eigenvalue of the linearized operator at $u=0$. Furthermore, they prove that, for any nonnegative, nontrivial initial density $u_0$, $u(t,x+ct)$ converges, as $t\rightarrow \infty$ to the unique nontrivial solution of the traveling wave equation 
\[ v'' + cv' +  f(x,v)   =0 ~~ \text{ in } \R.\]
There have been many extensions to the results of \cite{B+09}, which have been surveyed by Wang, Li, Dong and Qiao in \cite{W+22}. See also \cite{ABR17} for a reaction-diffusion model with a nonlocal interaction given by the nonlinearity.

On the other hand, in the last decade there has been a great interest in reaction-diffusion equations for nonlocal operators. These operators present interesting features, the most salient of which is the nonexistence of traveling waves for isotropic equations, i.e. with $f$ independent of $x$. As proven by Cabré and Roquejoffre \cite{CR13}, solutions to 
\[u_t + (-\Delta)^s u = f(u)\]
do not propagate with constant speed but accelerate, and their level sets travel exponentially in time.

The acceleration feature is shared by models of nonlocal dispersion, that is, diffusion driven by operators of type 
\[\mathcal{M}[u] = J \star u -u \]
where $J$ is a continuous probability density in $\R$. Reaction-diffusion equations involving these operators have a rich theory, we refer the reader to the recent preprint by Bouin and Coville \cite{BC25}, the introduction of which offers a rather complete overview of existing results. 

The moving patch model with nonlocal dispersion has been studied by Coville in \cite{C21}. For compactly supported kernels $J$, persistence is obtained for speeds under a parameter $c^*$ as well as extinction for speeds above another $c^{**}$. Using an approximation procedure, the persistence result is extended to kernels which are merely integrable, and the extinction one to kernels with finite second moment.

In this article, we tackle the question for the fractional Laplacian, as a natural continuation of the study in \cite{CR13} to nonlinearities with nonconstant $x$-dependence. This operator differs from those treated in \cite{C21} in two aspects: on the one hand, it is a singular integral operator, so it cannot be written in convolution form, and on the other hand, its kernel presents slowly decaying tails of order $|x|^{-1-2s}$. The main novelty of this work is that we treat reaction-diffusion equations with fractional diffusion in advective environments for the first time, obtaining existence, nonexistence and decay results for the corresponding traveling wave equation, as well as studying the long-time behavior of the evolution equation.

\subsection{Assumptions and main results}

We will make the following assumptions on $f$, which are identical to those in \cite{C21}. First, we assume that $f$ is of KPP-type, in the following sense
\begin{equation}
\tag{H\arabic{fhyp}}\stepcounter{fhyp}\label{hyp:KPP}
\left\{
\begin{aligned}
& f \text{ is differentiable with respect to } u; \\
& \forall u \geq 0, \ f(\cdot,u) \in C^{0,1}(\mathbb{R}); \\
& f(\cdot,0) \equiv 0; \\
& \text{for all } x \in \mathbb{R}, \ f(x,u)/u \text{ is decreasing with respect to } u \text{ on } (0,+\infty); \\
&\text{there exists } S(x) \in C(\mathbb{R}) \cap L^\infty(\mathbb{R}) \text{ such that } f(x,S(x)) \leq 0 \text{ for all } x \in \mathbb{R}.
\end{aligned}
\right.
\end{equation}
The second assumption expresses the environment as inadvantageous outside some compact set, that is
\begin{equation}
\tag{H\arabic{fhyp}}\stepcounter{fhyp}\label{hyp:bad_or}
\limsup_{x\rightarrow \infty} \frac{f(x,s)}{s} < 0~~ \text{ uniformly in }s \in \R.
\end{equation}
Note that this assumption combined with \eqref{hyp:KPP} implies $\limsup_{|x|\rightarrow \infty} \partial_u f(x,0) < 0$. Finally, we will need further regularity on $\partial_uf$, given by
\begin{equation}
\tag{H\arabic{fhyp}}\stepcounter{fhyp}\label{hyp:lipschitz}
\sup_{x\in\R}\left(\sup_{u_0,u_1 \geq 0} \frac{|\partial_uf(x,u_0) - \partial_uf(x,u_1)|}{|u_0-u_1|}\right) < \infty.
\end{equation}

The model case we have in mind is given by the nonlinearity $f(x,u) = u(a(x)-u^p)$ with $a\in L^{\infty}(\R)$ such that $\lim_{|x|\rightarrow \infty} a(x) = -\nu<0$ and $p \geq 1$. In this case $S(x) \equiv S >0$ satisfies the fifth condition in \eqref{hyp:KPP} whenever $S> \|a\|_{L^{\infty}(\R)}^{\frac{1}{p}}$. 

Throughout this note, we assume $s\in (\frac12, 1)$ and adopt the notation $\Delta^s = -(-\Delta)^s$, as well as
\[\Lop_{c,a} \phi(x) = \Delta^s \phi(x) + c\phi'(x) + a(x)\phi(x).\]

Our main results are the following. First, we obtain necessary and sufficient existence conditions for the traveling wave equation associated to our model:
\begin{mainthm}
\label{thm:tw}
Let $c\in\R$ and denote $\lambda = \lambda_1(\Delta^s + c\partial_x + \partial_u f(\cdot, 0))$. Then, there exists a unique nontrivial solution $u\in W^{1,\infty}(\R) \cap H^1(\R)$ to the problem
\begin{equation}
\tag{TW}\label{eq:tw}
\Delta^s u + cu' + f(x,u) = 0,
\end{equation}
if and only if $\lambda < 0$. In that case, then there exists $C>0$ such that 
\[ \frac{C^{-1}}{|x|^{1+2s}} \leq u(x) \leq \frac{C}{|x|^{1+2s}}.\]
\end{mainthm}

The definition and properties of $\lambda_1(\Lop_{c,a})$ will be discussed in section \ref{sect:eig}. A key difference between our setting and the second order case is the unavailablity of a straightforward way of describing $\lambda_1$ with respect to $c$. Bypassing this obstacle through suitable barriers and continuity arguments, we are able to establish a critical speed under which traveling waves exist, and a critical speed over which no nontrivial solutions to \eqref{eq:tw} exist.
\begin{mainthm}
\label{thm:thresholds}
Let $a\in C(\R) \cap L^{\infty}(\R)$ such that $\lambda_1(\Lop_{0,a})< 0$ and $\limsup_{|x|\rightarrow \infty} a(x) <0$, then there exist $c^*, c^{**}>0$ such that $\lambda_1(\Lop_{c,a}) < 0$ for all $c$ such that $c< c^*$ and $\lambda_1(\Lop_{c,a}) > 0$ for all $c$ such that $|c|>c^{**}$.
\end{mainthm}

Having studied the traveling wave equation, we turn to the treatment of the long time behavior of solutions to the parabolic problem. We are able to establish the following persistence and extinction theorem.
\begin{mainthm}
\label{thm:convergence}
Let $u_0 \in C(\R)$ be a bounded nonnegative function. Let $u$ be the unique solution to 
\begin{equation}
\label{eq:parabolic}
\begin{cases}
\partial_t u = \Delta^su + f(x,u) & \text{ in } \R \times (0,\infty)\\
u(0,\cdot) = u_0 & \text{ in } \R.
\end{cases}
\end{equation} 
Denote $\lambda = \lambda_1(\Lop_{c,a})$ with $a = \partial_s f(x,0)$. Then
\begin{enumerate}[(i)]
\item if $\lambda \geq 0$, then $\|u(t, \cdot)\|_{L^{\infty}} \rightarrow 0$ as $t\rightarrow \infty$;
\item uf $\lambda < 0$ and $u_{\infty}$ is the unique nontrivial solution to \eqref{eq:tw}, then $\|u(\cdot -ct) - u_{\infty}\|_{L^{\infty}} \rightarrow 0$ as $t\rightarrow \infty$.
\end{enumerate}
\end{mainthm}

The outline of the paper is the following. In Section \ref{sect:preliminaries} we state preliminary results regarding \eqref{eq:tw}, as well as the construction of a useful barrier. In Section \ref{sect:a_priori}, we investigate regularity and decay properties of solutions to equation \eqref{eq:tw}. In Section \ref{sect:eig} we state some of the properties of the principal eigenvalue of our linearized operator, as well as its qualitative behavior with respect to the drift $c$. In Section \ref{sect:existence} we complete the proof of Theorem \ref{thm:tw}. Lastly, in Section \ref{sect:parabolic} we prove Theorem \ref{thm:convergence}.

\section{Preliminaries}
\label{sect:preliminaries}

We first recall the definition of the fractional Laplacian. Given $s\in (0,1)$ and for any bounded, smooth function $\phi$ in $\R$, we define its fractional Laplacian at $x$ by
\[(-\Delta)^s \phi(x) = C_{s} ~ \PV \int_{\R} \frac{\phi(x)-\phi(y)}{|x-y|^{1+2s}} \dy \]
where $\PV$ stands for Cauchy principal value. The constant $C_{s} >0$ is chosen in order to have $\lim_{s\nearrow 1} (-\Delta)^s\phi(x) = -\Delta\phi(x)$ and has the explicit expression 
\[C_s = \frac{2^{2s}s\Gamma(\frac{1}{2} + s)}{\pi^{1/2} \Gamma(1-s)},\] 
see \cite{BV16} for a derivation of this formula as well as an introduction to this operator. Nonetheless, we will omit this constant for the sake of readability. 

Note that the above definition can be extended to functions $\phi$ which are $C^{2s}$ around $x$ and such that 
\[\|\phi\|_{\Lomega} = \int_{\R} \frac{|\phi(y)|}{1+|x|^{1+2s}} \dy < \infty.\]
Throughout this article, with ``solution" we mean classical solution. By ``viscosity solution" or ``the viscosity sense" we mean the definition given in e.g. \cite{CS09}, and use it to extend the definition of our operators to functions which are merely continuous and integrable.

\begin{defn}
Let $\Omega\subset \R$ be an open set. We say that $u\in C(\R) \cap \Lomega$ is a viscosity subsolution (supersolution) to 
\[\Delta^s u + cu' + f(x,u) = 0 ~~\text{at }x\in \Omega\]
if for any $\phi \in C^2(B_r(x))$, $B_r(x)\subset \Omega$, such that $\phi \geq (\leq) u$ in $B_r(x)$, $\phi(x) = u(x)$ we have
\[\Delta^s \phi_u(x) + c\phi_u'(x) + f(x, u(x)) \geq ~(\leq) 0\]
where $\phi_u$ denotes the extension of $\phi$ by $u$ outside $B_r(x)$. We say that $u$ is a sub- or supersolution in $\Omega$ if it is sub- or supersolution at every point in $\Omega$.

We say that $u$ is a vicosity solution if it is both a sub- and a supersolution.
\end{defn}

In this section, we state some known properties of the fractional Laplacian. First, a version of the maximum principle in unbounded domains.
\begin{prop}
\label{prop:mp}
Let $\Omega\subset \R$ an unbounded domain, $c\in \R$, $\nu\geq 0$. Let $w$ be a classical solution to 
\[\begin{cases}
\Delta^s w + c w' - \nu w \leq 0 & \inSet{\Omega} \\
w \geq 0 & \inSet{\R \setminus \Omega} \\
\lim_{|x|\rightarrow \infty} w \geq 0 &
\end{cases}\]
then $w \geq 0$ in $\R$.
\end{prop}

\begin{proof}
The proof is rather standard. If $w < 0$ somewhere, as $w$ is nonnegative at infinity, we can take $x_0$ a negative minimum of $w$. Thus, $w'(x_0) = 0$ and 
\begin{align*}
\Delta^sw(x_0) = \PV \int_{\R} \frac{u(y) - u(x_0)}{|x-y|^{1+2s}} \dy > 0.
\end{align*}
This implies that
\[0 \geq \Delta^s w(x_0) + c w'(x_0) - \nu w(x_0) \geq \Delta^s w(x_0) > 0,\]
a contradiction.
\end{proof}

Next we state the weak and strong Harnack inequalities. The weak version is from \cite{FN25} (see Remark 2 therein) and gives an explicit dependence on the radius.
\begin{prop}
\label{prop:weak_harnack}
Let $R> 0$ and $b,c \in C(B_{R+1} \setminus B_R) \cap L^{\infty}(B_{R+1}\setminus B_R)$. There exist a constant $C>0$ depending on $\|c\|_{L^{\infty}}$, $\|b\|_{L^{\infty}}$ and $s$ such that, if $v$ is a solution to  
\[\begin{cases}
\Delta^s v + c v' + bv \leq 0 & \inSet{B_{R+1} \setminus B_R} \\ 
v\geq 0 & \inSet{\R}
\end{cases}\]
then
\[\|v\|_{\Lomega} \leq C(1+R^{1+2s}) \inf_{B_{R+\frac34} \setminus B_{R+\frac14}}v.\]
\end{prop}

The strong Harnack inequality is given in the unit ball, as we will not be using its dependence on the radius.
\begin{prop}[Theorem 3.2 in \cite{DQT23}]
\label{prop:harnack}
Let $b,c, f \in L^{\infty}( B_1)$. Let $u \geq 0$ be a continuous function in $\mathbb{R}^N$ and suppose that
\[
\Delta^s u + c u + bu = f
\]
in the viscosity sense. Assume that $s > 1/2$, then there exists a constant $C$ such that
\[
u(x) \leq C \left(u(0) + \|f\|_{L^{\infty}} \right) \quad x \in B_{1/2}.
\]
\end{prop}

We finish this section by constructing a function which will serve as a useful barrier throughout the text. The following lemma is an adaptation of Lemma 4.3 of \cite{DQT23}, here we extend it to the range $s>\frac12$. 
\begin{lem}
\label{lem:trunc_fs}
Let $\kappa > 0$, $\beta \in (1,1+2s]$. Define $\Phi(x) = \min\{\kappa|x|^{2s-1}, |x|^{-\beta}\}$ and denote $r_{\kappa} = \kappa^{-\frac{1}{2s-1+\beta}}$. Then
\[\Delta^s\Phi(x) \leq \begin{cases}
\frac{(\kappa -1) |x|^{2s-1}}{2s-1}\left( \frac{1}{|x+r_{\kappa}|^{2s}} + \frac{1}{|x-r_{\kappa}|^{2s}} \right) & 0<|x|<r_{\kappa} \\
C(\beta)\kappa^{\theta}\Phi(x) & |x| > r_{\kappa}
\end{cases}\]
where $\theta = \frac{2s}{2s-1+\beta}$. 
\end{lem}

\begin{proof}
We will compute it directly. If $0 < |x| < r_{\kappa}$:
\begin{align*}
\Delta^s \Phi(x) &= \kappa \PV \int_{-r_{\kappa}}^{r_{\kappa}} \frac{|y|^{2s-1} - |x|^{2s-1}}{|x-y|^{1+2s}} \dy +  \int_{\R\setminus (-r_{\kappa}, r_{\kappa})} \frac{\Phi(y) - |x|^{2s-1}}{|x-y|^{1+2s}}\dy\\
&=\kappa \Delta^s(|x|^{2s-1}) + \int_{\R\setminus(-r_{\kappa},r_{\kappa})} \frac{\Phi(y) - \kappa|y|^{2s-1}}{|x-y|^{1+2s}}\dy + (\kappa-1)|x|^{2s-1} \int_{\R\setminus(-r_{\kappa},r_{\kappa})} \frac{\dy}{|x-y|^{1+2s}} \\
&\leq \int_{\R\setminus(-r_{\kappa},r_{\kappa})} \frac{\Phi(y) - \kappa|y|^{2s-1}}{|x-y|^{1+2s}}\dy + \frac{(\kappa -1) |x|^{2s-1}}{2s-1}\left( \frac{1}{|x+r_{\kappa}|^{2s}} + \frac{1}{|x-r_{\kappa}|^{2s}} \right)
\end{align*}
since $\Delta^s(|x|^{2s-1}) =0$, $x\mapsto |x|^{2s-1}$ being the fundamental solution to the fractional Laplacian in $\R$ (see \cite{CS07}). and $\Phi(y) - \kappa|y|^{2s-1}\leq 0$ outside $(-r_{\kappa}, r_{\kappa})$.

If $|x|>r_{\kappa}$:
\begin{align*}
\Delta^s \Phi(x) &= \PV \int_{\R\setminus (-r_{\kappa}, r_{\kappa})} \frac{|y|^{-\beta} - |x|^{-\beta}}{|x-y|^{1+2s}} \dy +  \int_{-r_{\kappa}}^{r_{\kappa}} \frac{\Phi(y) - |x|^{-\beta}}{|x-y|^{1+2s}}\dy\\
&= \Delta^s(|x|^{-\beta}) + \int_{-r_{\kappa}}^{r_{\kappa}} \frac{\Phi(y) - |y|^{-\beta}}{|x-y|^{1+2s}}\dy \\
&\leq C(\beta) |x|^{-2s}\Phi(x) \\
&\leq C(\beta) r_{\kappa}^{-2s} \Phi(x) \\
&= C(\beta) \kappa^{\frac{2s}{2s-1+\beta}}\Phi(x),
\end{align*}
which gives the desired formula.
\end{proof}

\begin{rmk}
We will use only $\beta = 1+2s$, in this case $r_{\kappa} = \kappa^{-\frac{1}{4s}}$ and $\theta=\frac12$.
\end{rmk}

\begin{cor}
\label{cor:supersol}
With the same hypothesis and notation as in Lemma \ref{lem:trunc_fs}, for any $\nu > 0$, there exist $\kappa_{\nu},~~R_{\nu}>0$ such that 
\[\Delta^s \Phi(x) + c\Phi'(x) - \nu \Phi(x) \leq 0, ~ \forall |x|>R_{\nu}.\]
\end{cor}

\begin{proof}
We want to find $\kappa$ and $R$ such that
\begin{subequations}
\label{eq:Phi_supersol}
\begin{align}
\Phi(x) &= |x|^{-\beta} ~~ \forall x\in \R \setminus (-R,R), \label{eq:Phi_value}\\
\Lop_{c,-\nu}\Phi = \Delta^s \Phi +c\Phi' -\nu \Phi  &\leq 0 \hspace{23pt} \forall x\in \R \setminus (-R,R), \label{eq:Phi_ineq}
\end{align}
\end{subequations}
for some $R\geq R_0$ fixed. To have \eqref{eq:Phi_value}, one must have $r_{\kappa} < R$, that is $\kappa> R^{-\frac{1}{2s-1+\beta}}$. With this in mind, we compute,
\begin{align*}
\Lop_{c,-\nu}\Phi(x) &\leq \Phi(x) \left(-\nu -\frac{\beta c}{x} +C(\beta)\kappa^{\theta}\right) \\
&\leq \Phi(x) \left(-\frac{\nu}{2} + C(\beta) \kappa^{\theta} \right),
\end{align*}
for any $|x| \geq R_1$, with $ R_1= \frac{2C\beta}{\nu}$.

Thus, in order to \eqref{eq:Phi_value} and \eqref{eq:Phi_ineq} to hold, one must have
\begin{equation}
\label{eq:restriction}
\frac{1}{R^{\frac{1}{2s-1+\beta}}} \leq \kappa \leq\left( \frac{\nu}{2C(\beta)} \right)^{\frac1{\theta}}.
\end{equation}
This is possible because $R^{\frac{\theta}{2s-1+\beta}}\frac{C(\beta)}{\nu}  > 1$ for $R\geq R_2$. Then, take $R_0 = \max\{R_1, R_2\}$ and any $\kappa$ satisfying \eqref{eq:restriction}.
\end{proof}

\section{A priori estimates}
\label{sect:a_priori}

In this section we obtain regularity and decay properties of solutions to the equation \eqref{eq:tw}. These will be essential for the existence and nonexistence results in the following sections.

\begin{prop}
\label{prop:regularity}
Let $u$ be a nonnegative bounded solution to \eqref{eq:tw}. Then
\[\|u\|_{C^{2s+\alpha}(\R)} \leq C(\|u\|_{\infty} + \|f(x,u) \|_{C^{0,\alpha}(\R)})\]
for some $\alpha \in (0,1)$.
\end{prop}

\begin{proof}
Let $x \in \R$ and consider $B = B_1(x)$. By the $C^{1,\alpha}$ interior regularity of the operator $\Delta^s + c\partial_x$, we have
\[|u'(y)| \leq C(|c|)(\|u\|_{L^{\infty}(\R\setminus B_1(x))}+ \|f(x,u)\|_{L^{\infty}(B_1(x))}) \leq C_0(\|u\|_{L^{\infty}(\R)}+ \|f(x,u)\|_{L^{\infty}(\R)}), ~~\forall y\in B. \]
As the right hand side does not depend on $x$, we have $u'\in L^{\infty}(\R)$.

By Proposition 2.9 in \cite{S07}, we have, for any $\beta < 2s - 1$,
\[\|u\|_{C^{1,\beta}(\R)} \leq C(\|u\|_{\infty} + \|f(x,u)\|_{\infty} + c\|u'\|_{\infty}).\]
Using well known interpolation inequalities, see e.g. chapter 6 in \cite{GT}, for any $\epsilon >0$, it holds 
\[\|u\|_{C^{1,\beta}(\R)} \leq \epsilon \|u\|_{C^{1,\beta}} + C_{\epsilon}(\|u\|_{\infty} + \|f(x,u)\|_{\infty}),\]
which yields, taking $\epsilon = \frac12$,
\begin{equation}
\label{eq:regularity}
\|u\|_{C^{1,\beta}(\R)} \leq C(\|u\|_{\infty} + \|f(x,u)\|_{\infty}).
\end{equation}

Fix $\beta \in (0, 2s-1)$, and take $\alpha = \min\{\beta, \gamma\}$, where $\gamma\in (0,1)$ is such that $f(x,u)$ is $\gamma$-Hölder continuous. Now, Proposition 2.8 in \cite{S07} implies
\[\|u\|_{C^{2s+\alpha}(\R)} \leq C(\|u\|_{\infty} + \|f(x,u) \|_{C^{0,\alpha}(\R)} + c \|u'\|_{C^{0,\alpha}}).\]
Applying \eqref{eq:regularity} to the last inequality yields the result.
\end{proof}

The next lemma goes in the spirit of Lemma 3.2 in \cite{MRS14}. 
\begin{lem}
\label{lem:frac_lap_bound}
Let $v\in C^{1,	\alpha}(\R) \cap W^{1,\infty}(\R)$ for some $\alpha \in (2s-1, 1)$. Then 
\[\sup_{R>0} \int_{-R}^R \Delta^s v(x) \dx \leq C(\|v\|_{\infty} + \|v'\|_{\infty})\]
\end{lem}

\begin{proof}
First, assume $v\in C^2(\R)$. Let $R>0$, thus
\begin{align*}
\int_{-R}^R \Delta^s v(x) \dx &= \int_{-R}^R \int_{B_1^c} \frac{v(x + y) - v(x)}{|y|^{1+2s}} \dy \dx + \int_{-R}^R \PV \int_{B_1} \frac{v(x+y) - v(x)}{|y|^{1+2s}} \dy \\
&= I_1 + I_2
\end{align*}

We bound each $I_i$ separately. First, by the fundamental theorem of calculus (FTC) we write
\[v(x+y) - v(x) =\int_0^1v'(x+ty)y\dt,\]
thus
\begin{align*}
I_1 &= \int_{-R}^R \int_{B_1^c} \frac{v(x + y) - v(x)}{|y|^{1+2s}} \dy \dx \\
&= \int_{-R}^R \int_{B_1^c} \int_0^1 \frac{ v'(x + ty)y}{|y|^{1+2s}} \dt\dy\dx \\
&= \int_{B_1^c} \frac{y}{|y|^{1+2s}}\int_0^1 \int_{-R}^R v'(x + ty) \dx \dt \dy \\
&= \int_{B_1^c} \frac{y}{|y|^{1+2s}}\int_0^1 (v(ty + R) - v(ty - R))\dt \dy \\
&\leq 2\|v\|_{\infty} \int_{B_1^c} \frac{1}{|y|^{2s}}\dy < \infty,
\end{align*}
because $2s> 1$. 

On the other hand, by applying the FTC to the function $ g(r) = \int_0^1 v'(x+try)y\dt$ we have
\[v(x+y) - v(x) -v'(x)y = \int_0^1 v'(x+ty) y - v'(x)y \dt = \int_0^1 \int_0^1 v''(x +try) ty^2 \dr \dt\]
which yields, recalling that $\int_{-1}^1 v'(x) y \dy = 0$, 
\begin{align*}
I_2 &= \PV\int_{-R}^R \int_{B_1} \frac{v(x+y) - v(x) - v'(x)y}{|y|^{1+2s}} \dy \dx\\
&= \int_{-R}^R \int_{B_1} \int_0^1 \int_0^1 \frac{ty^2 v''(x+try)}{|y|^{1+2s}} \dr\dt \dy \dx \\
&= \int_{B_1} \frac{1}{|y|^{2s-1}}\int_0^1 \int_0^1 t \int_{-R}^R v''(x+try) \dx  \dr\dt \dy \\
&= \int_{B_1} \frac{1}{|y|^{2s-1}}\int_0^1 \int_0^1 t (v'(try + R) - v' (try -R)) \dr\dt \dy \\
&\leq   \|v'\|_{\infty}  \int_{B_1} \frac{1}{|y|^{2s-1}} \dy < \infty.
\end{align*}

Now, for the general case, we perform a standard approximation argument. Let $(v_k) \subset W^{1,\infty}(\R) \cap C^2(\R)$ such that $v_k\rightarrow v$ in $C_{\text{loc}}^{1,\alpha}(\R)$ and in $W^{1,\infty}(\R)$. Now, let $R>0$. We have
\[\left| \int_{-R}^R \Delta^s v\dx \right| \leq C\left| \int_{-R}^R \Delta^s v_k\dx \right| + \left| \int_{-R}^R \Delta^s(v-v_k) \dx\right| \]

Let us denote $w_k := v-v_k$. Then
\begin{align*}
\left|\int_{-R}^R \int_{\R\setminus(-1,1)} \frac{w_k(x+y) - w_k(x)}{|y|^{1+2s}} \dy\dx \right| &\leq \frac{2R\|w_k\|_{\infty}}{s}
\end{align*}
and, by the FTC once more,
\begin{align*}
\left|\int_{-R}^R \PV \int_{-1}^{1} \frac{w_k(x+y) - w_k(x)}{|y|^{1+2s}} \dy\dx \right| &=   \frac12 \left|\int_{-R}^R  \int_{-1}^{1} \frac{w_k(x+y) + w_k(x-y) - 2w_k(x)}{|y|^{1+2s}} \dy\dx \right|\\
&=  \frac12 \left|\int_{-R}^R  \int_{-1}^{1} \int_0^1 \frac{(w_k'(x+ty) - w_k'(x-ty)y)}{|y|^{1+2s}} \dt \dy\dx \right|\\
&\leq \frac12 \int_{-R}^R  \int_{-1}^{1}\int_0^1 \frac{|w_k'(x+ty) - w_k'(x-ty)|}{|y|^{2s}} \dt\dy\dx \\
&\leq \frac1{2^{1-\alpha}} \int_{-R}^R  \int_{-1}^{1}\int_0^1 \frac{[w_k']_{[-R-1, R+1]} t^{\alpha}}{|y|^{2s-\alpha}} \dt\dy\dx \\
&= \frac{2^{1+\alpha}R[w_k']_{[-R-1, R+1]}}{(1+\alpha)(\alpha-2s+1)}.
\end{align*}

For every $R>0$, we can find $k_R\in \N$ such that $[w_k']_{[-R-1, R+1]} < \frac1{R^2}$ for every $k>k_0$. This implies
\begin{align*}
\left| \int_{-R}^R \Delta^s v(x) \dx \right| \leq \left| \int_{-R}^R \Delta^s v_k(x) \dx \right| + \left| \int_{-R}^R \Delta^s w_k(x) \dx \right| \leq C\left(\|v_k\|_{\infty} + \|v_k\|_{\infty} + \frac{1}{R}\right)
\end{align*}
for some $C$ independent of $R$. Taking the supremum over $R>0$ we will have $k\rightarrow \infty$, and the lemma follows.
\end{proof}

\begin{prop}
\label{prop:decay}
Let $s \in (\frac12, 1)$. Let $v\in C^2(\R)$ be a bounded, nontrivial, nonnegative solution to \eqref{eq:tw}. Then,
\begin{equation}
\label{eq:decay}
\frac{C^{-1}}{|x|^{1+2s}} \leq v(x) \leq \frac{C}{|x|^{1+2s}},\medskip \forall |x| > R
\end{equation}
for some constants $C, R > 0$.
\end{prop}

\begin{proof} The lower bound in \eqref{eq:decay} is direct from the weak Harnack inequality Proposition \ref{prop:weak_harnack}. For the upper bound, we divide the proof in two steps.

\textit{Step 1:} We will now prove that $v \in L^1(\R)$, which by smoothness of $v$ implies that $\lim_{|x| \rightarrow \infty} v(x) = 0$. The argument is inspired in the proof of Proposition 4.6 in \cite{C21}.

Integrate \eqref{eq:tw} over $[-R,R]$ and get
\[\int_{-R}^R \Delta^s v \dx + c(v(R) - v(-R)) = \int_{-R}^R -f(x,v(x))\dx.\]
Take $\nu \in (0,1)$ and $R_0>0$ such that $f(x,s) < -\nu s$ for all $|x| > R_0$. By Proposition \ref{prop:regularity}, $v'\in L^{\infty} \cap C^{0, \beta}$ for some $\beta \in (2s-1, 1)$, so we can apply Lemma \ref{lem:frac_lap_bound} to get, for $R> R_0$,
\begin{align*}
 \int_{[-R, R] \setminus [-R_0, R_0]} \nu v(x) \dx \leq  \int_{[-R, R] \setminus [-R_0, R_0]} -f(x,v(x)) \dx& \leq 2|c|\|v\|_{\infty} + C(\|v'\|_{\infty} + \|v\|_{\infty}) + \int_{-R_0}^{R_0} f(x,v(x)) \dx \\
\Rightarrow \int_{-R}^R v(x) \dx & \leq \frac{1}{\nu}\left(C(\|v'\|_{\infty} + \|v\|_{\infty}) + \int_{-R_0}^{R_0}[ f(x,v(x)) + v(x)] \dx\right).
\end{align*}

To conclude, note that, as $v\geq 0$ and the right hand side does not depend on $R$, we can take $R\rightarrow \infty$ and get $\|v\|_{L^1} < \infty$.

\textit{Step 2:} We prove the decay estimate \eqref{eq:decay}. Defining $R_0$ and $\nu$ as in step 1, we have 
\[\Delta^s v +cv' -\nu v \geq \Delta^s v + c v' + f(x,v) \geq 0 ,\text{ for } |x|>R_1.\]

Our goal is to compare $v$ with a suitable barrier. Let $\Phi$ be the function from Corollary \ref{cor:supersol} and Remark 1. 

Now, take $\epsilon \in (0, \frac{r_{\kappa}}{2})$ and $\eta_{\epsilon}$ a regularizing kernel with support in $B_{\epsilon}$. Denoting $\Phi_{\epsilon} = \eta_{\epsilon} * \Phi$, we have 
\[\inf_{[-R,R]} \Phi_{\epsilon}(x) = m > 0,\]
so taking $w = \frac{\|v\|_{\infty}}{m} \Phi_{\epsilon}$ we have that $w$ is a supersolution to $\Delta^s w + cw' -\nu w \leq 0$ in $\R \setminus [-R, R]$ such that $w \geq v$ in $[R,R]$. As $\lim_{|x| \rightarrow \infty} v(x) = \lim_{|x| \rightarrow \infty} w(x) = 0$, we can apply Proposition \ref{prop:mp} to $u = w-v$ in $\R \setminus [-R,R]$ to get $v\leq w$ in $\R$. Thus,
\begin{align*}
v(x) &\leq \frac{\|v\|_{\infty}}{m}\int_{B_{\epsilon}} \eta_{\epsilon}(y) \Phi(x-y) \dy
\\
&= \frac{\|v\|_{\infty}}{m}\int_{B_{\epsilon}} \frac{\eta_{\epsilon}(y)}{|x-y|^{1+2s}}\dy \\
&\leq \frac{\|v\|_{\infty} |B_{\epsilon}|}{m} \frac{1}{(|x|+\epsilon)^{1+2s}},
\end{align*}
which gives \eqref{eq:decay}.
\end{proof}

Next we prove decay on the derivative of the solution.

\begin{prop}
\label{prop:H1}
If $u$ is a bounded solution to \eqref{eq:tw}, then $u \in H^1(\R)$.
\end{prop}

\begin{proof}
Multiply \eqref{eq:tw} by $u'$ and integrate on $(-R,R)$ to get
\[c\int_{-R}^R (u')^2 \dx = \int_{-R}^R - u'f(x,u) + u'(-\Delta)^su dx\]
The first term on the right-hand side is bounded by
\begin{align*}
\left|\int_{-R}^R - u'f(x,u) \right| &\leq \|u'\|_{\infty} \int_{-R}^R \left|\frac{f(x,u)}{u}\right| u\dx \\
&\leq \|u'\|_{\infty} \|u\|_{L^1} \sup_{x\in \R} \left|\frac{f(x,u(x))}{u(x)}\right| \\
&\leq \|u'\|_{\infty} \|u\|_{L^1} \|\partial_s f(\cdot, 0)\|_{\infty}.
\end{align*}

Using the integration by parts formula for the fractional Laplacian, given by Lemma 3.3 in \cite{DRV17}, the second term becomes 
\begin{equation}
\label{eq:ipp}
\int_{-R}^R u'(-\Delta)^s u\dx = \int_{\R \setminus (-R,R)} u' \Ns u - \int_{\R^2 \setminus (\R \setminus (-R,R)^2)} \frac{(u'(x) - u'(y)) (u(x) - u(y))}{|x - y|^{1+2s}}\dx \dy
\end{equation}
where 
\[\Ns u (x) = \int_{-R}^R \frac{u(x)-u(y)}{|x-y|^{1+2s}} \dy ~\text{ for }|x| > R. \]
Here every integral converges because $u\in C^{2s+\alpha}(\R) \cap W^{1,\infty}(\R)$ by Proposition \ref{prop:regularity}. 

We will estimate each term separately. For the Neumann term  $u'\Ns u$, let $\delta > 0$ be small but fixed. Then
\begin{align*}
\int_{\R \setminus (-R,R)} u' \Ns u  = \int_{\R \setminus (-(R+\delta), R+\delta))} u' \Ns u + \int_{(-R-\delta, -R) \cup (R, R+\delta)} u'\Ns u = I_T + I_B.
\end{align*}

The tail term $I_T$ will be controlled by the integrability of $\Ns u$. Let $|x| > R+ \epsilon$, then
\begin{align*}
|\Ns u(x)| &= \left|\int_{-R}^R \frac{u(x) - u(y)}{|x-y|^{1+2s}} \dy \right| \\
&= \left|u(x) \int_{-R}^R \frac{\dy}{|x-y|^{1+2s}} - \int_{-R}^R \frac{u(y)}{|x-y|^{1+2s}} \dy \right|\\
&\leq 2\|u\|_{\infty} \left| \int_{-R}^R \frac{1}{|x-y|^{1+2s}} \dy \right| \\
&= \frac{\|u\|_{\infty}}{s} \left| \frac{1}{|x+R|^{2s}} - \frac{1}{|x-R|^{2s}} \right|
\end{align*}
which yields
\begin{align*}
\left|I_T \right| &\leq \|u'\|_{{\infty}} \left( \int_{-\infty}^{-R-\delta} |\Ns u(x)| \dx +\int_{R+\delta}^{\infty} |\Ns u(x)| \dx \right) \\
&\leq \frac{2\|u'\|_{\infty} \|u\|_{\infty}}{s(2s-1)}\left(\frac{1}{\delta^{2s-1}} - \frac{1}{(2R + \delta)^{2s-1}} \right).
\end{align*}

The term around $-R$ and $R$, $I_B$, is controlled as follows. For $x\in (-R-\delta, -R)\cup (R, R+\delta)$, one has
\begin{align*}
|\Ns u(x)| &= \left| \int_{-R}^{R} \int_0^1 \frac{u'(y + t(x-y))(x-y)}{|x-y|^{1+2s}}\dt \dy\right| \\
&\leq \|u'\|_{L^{\infty}} \int_{-R}^R \frac{1}{|x-y|^{2s}} \dy\\
&= \frac{\|u'\|_{L^{\infty}}}{2s-1} \left| \frac{1}{|x+R|^{2s	-1}} - \frac{1}{|x-R|^{2s-1}}\right|.
\end{align*}
Thus,
\begin{align*}
|I_B| &\leq \frac{\|u'\|_{\infty}^2}{2s-1}\int_{(-R-\delta, -R) \cup (R, R+\delta)}\left| \frac{1}{|x+R|^{2s	-1}} - \frac{1}{|x-R|^{2s-1}}\right| \dx\\
&= \frac{\|u'\|_{\infty}^2}{(1-s)(2s-1)} (\delta^{2-2s} + (2R + \delta)^{2-2s} - (2R)^{2-2s}),
\end{align*}
which is bounded as $R\rightarrow \infty$, more precisely, 
\[\delta^{2-2s} + (2R)^{2-2s} - (2R + \delta)^{2-2s} = \delta^{2-2s} + o(1)\]
because, recalling that $2-2s\in (0,1)$, $z\mapsto + (2R)^{2-2s} - (2R + \delta)^{2-2s} $ is a positive decreasing function in $[0,\infty)$.

We have proven that 
\[\limsup_{R\rightarrow \infty} \left| \int_{\R \setminus (-R,R)} u' \Ns u  \right| \leq C\left(\frac{\|u'\|_{\infty} \|u\|_{\infty}}{\delta^{2s-1}} + \|u'\|_{\infty}^2 \delta^{2-2s}\right)\]

We turn to the second term in \eqref{eq:ipp}. We denote
\begin{align*}
Q_R := \R^2 \setminus (\R \setminus (-R,R))^2
\end{align*}
see figure \ref{fig:regions} for an illustration of this set. 

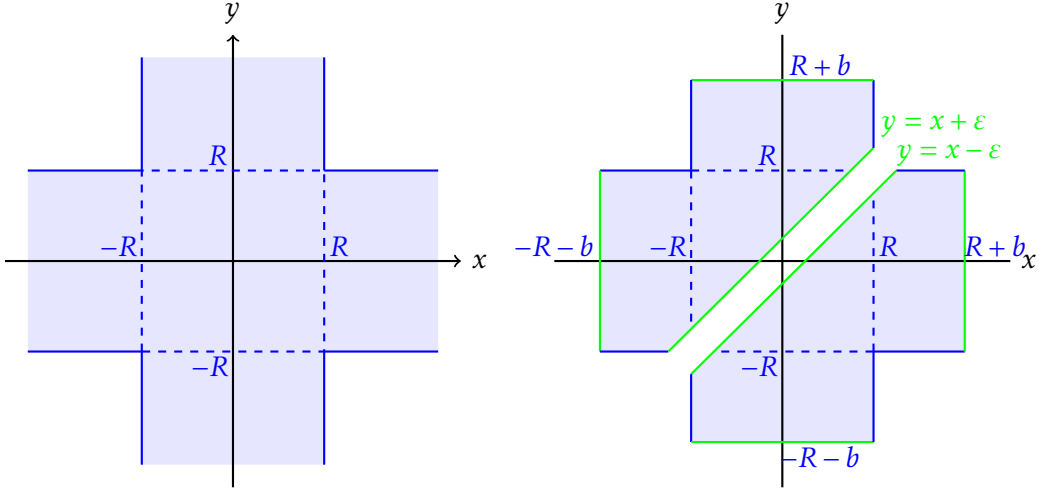
\begin{figure}
\label{fig:regions}
\begin{tikzpicture}[scale=1.0]
    \def\R{1.2}
    \def\L{2.7}
    
	\fill[blue!10] (-\L,\R) -- (-\R,\R) -- (-\R,\L) -- 
				(\R, \L) -- (\R, \R) -- (\L, \R) --
				(\L, -\R) -- (\R, -\R) --  (\R, -\L) --
				(-\R, -\L) -- (-\R, -\R) -- (-\L, -\R) -- 
				(-\L,\R);    
   
    \draw[thick,->] (-3,0) -- (3,0) node[right] {$x$};
    \draw[thick,->] (0,-3) -- (0,3) node[above] {$y$};
				
    \draw[blue,dashed,thick] (-\R,-\R) rectangle (\R,\R);

    \draw[blue,thick] (-\L,\R) -- (-\R,\R);
    \draw[blue,thick] (\R, \R) -- (\L,\R);
    \draw[blue,thick] (-\L,-\R) -- (-\R,-\R);
    \draw[blue,thick] (\R, -\R) -- (\L, -\R);
    \draw[blue,thick] (\R,-\L) -- (\R,-\R);
    \draw[blue,thick] (\R, \R) -- (\R, \L);
    \draw[blue,thick] (-\R,-\L) -- (-\R,-\R);
    \draw[blue,thick] (-\R, \R) -- (-\R, \L);
    
    \node[blue] at (\R+0.2,0.2) {$R$};
    \node[blue] at (-\R-0.3,0.2) {$-R$};
    \node[blue] at (-0.2,\R+0.2) {$R$};
    \node[blue] at (-0.3,-\R-0.2) {$-R$};
\end{tikzpicture}
\begin{tikzpicture}[scale=1.0]
    \def\R{1.2}
    \def\b{1.2}
    \def\e{0.3}
    
    \fill[blue!10] (\R+\b, \R) -- (\R+\b,-\R) -- (\R, -\R) -- 
    				(\R, -\R-\b) -- (-\R, -\R-\b) -- (-\R, -\R -\e) -- 
    				(\R+\e, \R) -- (\R+\b, \R);
    \fill[blue!10] (\R, \R + \b) -- (\R, \R + \e) -- (-\R-\e, -\R) -- 
    				(-\R-\b, -\R) -- (-\R-\b, \R) -- (-\R, \R) -- 
    				(-\R, \R+\b) -- (\R,\R+\b); 
    
    \draw[thick] (-3,0) -- (3,0) node[right] {$x$};
    \draw[thick] (0,-3) -- (0,3) node[above] {$y$};

    \draw[blue,thick] (\R,-\R) -- (\R+\b,-\R);
    \draw[green,thick] (\R+\b,-\R) -- (\R+\b,\R);    
    \draw[blue,thick] (\R+\b,\R) -- (\R+\e,\R);
    \draw[green,thick] (\R+\e,\R) -- (-\R, -\R-\e);
    \draw[blue,thick] (-\R, -\R-\e) -- (-\R, -\R-\b);
    \draw[green,thick] (-\R, -\R-\b) -- (\R, -\R-\b);
    \draw[blue, thick] (\R, -\R-\b) -- (\R, -\R); 

    \draw[blue,thick] (-\R,\R) -- (-\R-\b,\R);
    \draw[green,thick] (-\R-\b,\R) -- (-\R-\b,-\R);    
    \draw[blue,thick] (-\R-\b,-\R) -- (-\R-\e,-\R);
    \draw[green,thick] (-\R-\e,-\R) -- (\R, \R+\e);
    \draw[blue,thick] (\R, \R+\e) -- (\R, \R+\b);
    \draw[green,thick] (\R, \R+\b) -- (-\R, \R+\b);
    \draw[blue,thick] (-\R, \R+\b) -- (-\R, \R);

	\draw[blue, dashed, thick] (\R, -\R) -- (\R, \R-\e);
	\draw[blue, dashed, thick] (\R, -\R) -- (-\R+\e, -\R);
	\draw[blue, dashed, thick] (-\R, \R) -- (\R-\e, \R);
	\draw[blue, dashed, thick] (-\R, \R) -- (-\R, -\R+\e);
	
    \node[blue] at (\R+\b+0.4,0.2) {$R+b$};
    \node[blue] at (-\R-\b-0.6,0.2) {$-R-b$};
    \node[blue] at (0.5,\R+\b+0.2) {$R+b$};
    \node[blue] at (0.5,-\R-\b-0.2) {$-R-b$};
    
   	\node[blue] at (\R+0.2,0.2) {$R$};
    \node[blue] at (-\R-0.3,0.2) {$-R$};
    \node[blue] at (-0.2,\R+0.2) {$R$};
    \node[blue] at (-0.3,-\R-0.2) {$-R$};
    
    \node[green] at (2.0,1.5+\e) {$y = x + \varepsilon$};
    \node[green] at (2.0+0.2,1.5-\e+0.2) {$y = x - \varepsilon$};
    
\end{tikzpicture}
\caption{The set $Q_R$ and its truncation $Q_{R,b,\epsilon}$.}
\end{figure} 

Our strategy will be to apply the divergence theorem to simplify the second integral on the right hand side of \eqref{eq:ipp}. Note that
\begin{align*}
2\frac{(u'(x) - u'(y)) (u(x) - u(y))}{|x - y|^{1+2s}} =&~ \frac{\partial_x((u(x)- u(y))^2) + \partial_y((u(x)- u(y))^2)}{|x-y|^{1+2s}} \\
= &~ \partial_x\left(\frac{(u(x)- u(y))^2}{|x-y|^{1+2s}}\right) + (1+2s)\frac{\sgn(x-y)(u(x)- u(y))^2}{|x-y|^{2+2s}} \\
& + \partial_y\left(\frac{(u(x)- u(y))^2}{|x-y|^{1+2s}}\right) - (1+2s)\frac{\sgn(x-y)(u(x)- u(y))^2}{|x-y|^{2+2s}}\\
 = & ~\div(U(x,y))
\end{align*}
where we defined the vector field
\[U(x,y) := \frac{(u(x)- u(y))^2}{|x-y|^{1+2s}}(1,1).\]

Let $\epsilon>0$ be small and $b>0$ be large. We truncate $Q_R$ setting 
\[Q_{R, b, \epsilon} = Q_R \cap \{|x-y| >\epsilon\} \cap (-R-b, R+b),\] 
giving bounded regions as in figure \ref{fig:regions}. Note that $U \in C^1(\overline{Q_{R, b, \epsilon}})$, so we can apply the divergence theorem to get
\begin{align*}
2\int_{Q_{R,b,\epsilon}} \frac{(u'(x) - u'(y)) (u(x) - u(y))}{|x - y|^{1+2s}}\dx \dy =& \int_{\partial Q_{R, b, \epsilon}} U \cdot \hat{n} \mathrm{d}S.
\end{align*}
where $\hat{n}$ is the outer unit normal to $Q_{R,b,\epsilon}$.

We look at each segment of $\partial Q_{R,b,\epsilon}$ separately. For the segments intersecting $\partial(-(R+b), R+b))$ it suffices to see that, e.g., for $U = (U_1, U_2)$
\begin{align*}
\left|\int_{-R}^{R} U_1(R+b, y) \dy\right| &= \left|\int_{-R}^R \frac{(u(R+b)- u(y))^2}{|R+b - y|^{1+2s}}\dy \right|\\
&\leq \frac{2\|u\|_{L^{\infty}}^2}{s}\left( \frac{1}{b^{2s}} - \frac{1}{(2R + b)^{2s}} \right)
\end{align*} 
to conclude that each of those terms vanish as $b\rightarrow \infty$. 

The segments on the first and third quadrant cancel out, indeed, parametrizing e.g. the segments on the first quadrant yields
\begin{align*}
\int_{R+\epsilon}^{R+b} U_1(R, y) \dy - \int_{R+\epsilon}^{R+b} U_2(x, R) \dx &= \int_{R+\epsilon}^{R+b} \left( \frac{(u(R)- u(z))^2}{|R-z|^{1+2s}} - \frac{(u(R)- u(z))^2}{|R-z|^{1+2s}}\right)\dz = 0.
\end{align*}
The segments on the second and fourth quadrant are bounded since
\begin{align*}
\int_R^{R+b}\frac{(u(x) - u(-R))^2}{|R+x|^{1+2s}}\dx\leq \frac{2 \|u\|_{\infty}^2}{s} \left(\frac{1}{(2R)^{2s}} - \frac{1}{(2R + b)^{2s}}\right),
\end{align*}
thus
\[\lim_{b\rightarrow \infty} \int_R^{R+b}\frac{(u(x) - u(-R))^2}{|R+x|^{1+2s}}\dx \leq  \frac{2\|u\|_{\infty}}{s} \frac{1}{(2R)^{2s}},\]
and analogously for the other segments.

Finally, both the segments parallel to $\{x=y\}$ have normals orthogonal to $(1,1)$, so their contribution is 0.

Collecting all of our bounds we have
\[\limsup_{R\rightarrow \infty} \left|\int_{-R}^R u'(-\Delta)^s u\right| \leq  C\left(\frac{\|u'\|_{\infty} \|u\|_{\infty}}{\delta^{2s-1}} + \|u'\|_{\infty}^2 \delta^{2-2s}\right)\]
for some fixed $\delta > 0$, which implies
\begin{align*}
\lim_{R\rightarrow \infty}c\int_{-R}^R (u')^2 \leq  C\left(\frac{\|u'\|_{\infty} \|u\|_{\infty}}{\delta^{2s-1}} + \|u'\|_{\infty}^2 \delta^{2-2s}\right) + \|u'\|_{\infty} \|u\|_{L^1} \|\partial_s f(\cdot, 0)\|_{\infty}.
\end{align*}
This concludes the proof.
\end{proof}

\section{Eigenvalues}
\label{sect:eig}

In this section, we describe some properties of the principal eigenvalue for nonlocal operators. 
\begin{defn} 
Let $\Omega \subset \R$ be an open set, $c\in \R$ and $a \in C(\overline{\Omega})$. We define
\begin{align*}
\lambda_1(\Lop_{c,a}, \Omega):= \sup\{
\lambda \in \R:~ &\exists \psi \in C(\overline{\Omega}) \cap \Lomega, \psi >0 \text{ in } \Omega, \psi \geq 0 \text{ in } \R^N \setminus \Omega \\
& \text{such that } \Lop_{c,a} \psi + \lambda \psi \leq 0 \text{ in }\Omega, \text{ in the viscosity sense}\}
\end{align*}
\end{defn}

\begin{rmk}
When the domain is $\R$, we will simply write $\lambda_1(\Lop_{c,a}) := \lambda_1(\Lop_{c,a}, \R)$. And when $a = \partial_s f(x,0)$, we will write $\lambda_1 = \lambda_1(\Lop_{c,a}, \R)$.
\end{rmk}

\begin{rmk}
It is proven in \cite{QSX20} that, for smooth and bounded $\Omega$, $\lambda_1$ admits the characterization
\begin{align*}
\lambda_1(\Lop_{c,a}, \Omega):= \sup\{
\lambda \in \R:~~ -\Lop_{c,a+\lambda} \text{ satisfies the maximum principle in }\Omega \},
\end{align*}
and, since we have a linear operator, this characterization is still valid if one replaces the maximum principle by the minimum principle.
\end{rmk}

\begin{thm}
\label{thm:eigfunct_existence}
Let $ \Omega \subset \R$ be a bounded interval, $a \in C(\overline{\Omega})$. Then, there exists $\psi \in C^{\alpha}(\overline{\Omega})$ for some $\alpha \in (0,1)$ such that
\[\begin{cases}
\Lop_{c,a} \psi = -\lambda_1 \psi &\text{in } \Omega, \\
\psi >  0 & \text{in }  \Omega, \\
\psi = 0 & \text{in } \R^N \setminus \Omega.
\end{cases}\]
\end{thm}

\begin{proof}
Define $\tilde{\Lop}_{\sigma} := \Lop_{c,a} - \sigma$. For $\sigma \geq \sup_{\Omega} a$, the operator $(-\tilde{\Lop})^{-1}:C_0(\overline{\Omega}) \rightarrow C_0(\overline{\Omega})$ is well defined by the existence theorem of \cite{M17}, is compact by the Hölder regularity up to the boundary of \cite{QSX20}, and sends the closed convex cone of nonnegative functions into its interior by the weak and strong maximum principles (recall that $a-\sigma \leq 0$).

By the Krein-Rutman theorem, there exists $\lambda^* \in \R$ and $\psi\in C^{\alpha}(\overline{\Omega})$, $\psi >0$ in $\Omega$, such that
\[\begin{cases}
-\tilde{\Lop}_{\sigma}\psi = \lambda^* \psi & \text{en } \Omega \\ 
\psi = 0 & \text{en } \R^N \setminus \Omega.
\end{cases}\]
Arguing as in \cite{QSX20}, we conclude $\lambda_1 = \lambda^* - \sigma$. 
\end{proof}

The following properties are all easy consequences of the definition of $\lambda_1$, and we omit their proofs.

\begin{prop} 
\label{prop:eig_properties}
Let $\Omega, \Omega_1, \Omega_2 \subset\R$ be intervals, and $a, a_1, a_2$ be continuous functions. The following properties hold.  
\begin{enumerate}[(i)]
\item If $\Omega_1 \subset \Omega_2$, then $\lambda_1(\Lop_{c,a}, \Omega_1) \geq \lambda_1(\Lop_{c,a}, \Omega_2)$.
\item If $a_1 (x) \geq a_2(x)$ for all $x\in \Omega$, then $\lambda_1(\Lop_{c,a_1},\Omega) \leq \lambda_1(\Lop_{c,a_2}, \Omega)$.
\item $|\lambda_1(\Lop_{c,a_1}, \Omega) - \lambda_1(\Lop_{c,a_2}, \Omega)| \leq \|a_1 - a_2 \|_{\infty}$.
\item $\lambda_1(\Lop_{c,a}, \Omega) \geq -\sup_{\Omega} a$.
\end{enumerate}
\end{prop}

The principal eigenvalue in unbounded domains is a delicate subject even in the second order case. Nevertheless, our definition lets us conclude the existence of a positive eigenfunction. 

\begin{prop}
\label{prop:eigfunct_continuity}
Let $\Omega\subset \R$ be an interval and $a\in C(\overline{\Omega})\cap L^{\infty}(\R)$ let $(\Omega_n)_{n\in \N}$ be an increasing sequence of bounded intervals such that $\Omega_n\nearrow \Omega$. Then
\[\lim_{n\rightarrow\infty} \lambda_1(\Lop_{c,a},\Omega_n) = \lambda_1(\Lop_{c,a},\Omega).\]
Moreover, there exists a function $\phi \in C(\overline{\Omega}) \cap C^{\alpha}(\Omega)$ such that $ \Lop_{c,a} \phi =-\lambda_1(\Lop_{c,a},\Omega) \phi$ in $\Omega$ and $\phi\equiv 0$ in $\R\setminus \Omega$ when $\Omega \neq \R$.
\end{prop}

\begin{proof}
Denoting $\lambda_n:= \lambda_1(\Lop_{c,a}, \Omega_n)$, let us first prove that
\begin{equation}
\label{eq:leftineq}
\lim_{n\rightarrow\infty} \lambda_n \leq \lambda:= \lambda_1(\Lop_{c,a},\Omega).
\end{equation}
The proof follows the same lines as that of Theorem 1.1 in \cite{DQT23}: note first that
\[-\sup_{\R} a \leq \lambda_n \leq \lambda_0 = \lambda_1(\Lop_{c,a}, \Omega_0).\]
Since $\lambda_n$ is a nonincreasing sequence bounded from below, $\lambda^* := \lim_{n\rightarrow \infty} \lambda \in \R$ is well defined.
 
Let $\psi_n$ be the eigenfunction associated to $\lambda_n$, which exists because $\Omega_n$ is bounded. We choose $x_0 \in \Omega_0$, and normalize $\psi_n(x_0)= 1$. Using Proposition \ref{prop:harnack} and the fact that the coefficients are bounded, we get $\|\psi_n \|_{L^{\infty}(K)} \leq C$ for any $K\subset \R$ compact. By the regularity of nonlocal equations with a drift from \cite{CL12}, $\psi_n$ is bounded in $\Cloc^{1,\alpha}$, so it converges, up to a subsequence, locally uniformly and in $\Cloc^{1,\alpha}$ to some $\psi$. By the dominated convergence theorem, convergence also holds in $\Lomega$. 

Now, by the stability of viscosity solutions \cite{CL12}, $\psi$ must be a viscosity solution to
\[\begin{cases}
\Delta^s\psi + c \cdot \nabla \psi + (a + \lambda^*) \psi = 0 & \inSet{\Omega}\\
\psi > 0 & \text{ en } \Omega \\
\psi = 0 & \text{ en } \R^N \setminus \Omega.
\end{cases}\]
By definition of $\lambda_1(\Lop_{c,a},\Omega)$, this implies $\lambda^* \leq \lambda_1(\Lop_{c,a},\Omega)$. But as one has $\lambda_n \geq \lambda_1(\Lop_{c,a},\Omega)$, the proposition follows.
\end{proof}

The definition and the three propositions hold in $\R^N$ with minor modifications.

The next will be a key proposition in proving uniqueness and nonexistence for the traveling wave equation, as it concludes the existence of a subsolution to an adjoint equation. It is the adaptation of Proposition 6.1 in \cite{C21} to our case. Recall $\Lop_{\sigma,\beta} = \Delta^s v + \sigma v' + bv$ for $\sigma \in \R$ and $b\in C(\R) \cap L^{\infty}(\R)$.
\begin{prop}
\label{prop:subsol}
Suppose there exists $v\in L^1(\R)$ such that $v>0$ is a viscosity subsolution to 
\begin{equation}
\label{eq:subsol}
\Delta^s v(x) + c v'(x) + f(x,v(x)) \geq 0,~~\text{ for a.e. x in }\R;
\end{equation}
then, defining $b(x)=\frac{f(x,v(x))}{v(x)}$, we have $\lambda_1(\Lop_{c,b})\leq 0$.

Furthermore, there exists $\phi\in W^{1,\infty}(\R)$ such that
\[\Lop_{-c,b}\phi = \Delta^s \phi-c\phi' +b\phi \geq 0~~\text{ in }\R, \]
and such that $\Delta^s\phi \in L^{\infty}$.
\end{prop}

\begin{proof}
By contradiction, suppose $\lambda_1(\Lop_{c,b}) =: \lambda_1 > 0$. By definition, this implies that for any $\rho \in (0, \lambda_1)$ there exists $\psi$ such that $(\Lop_{c,b} + \rho) \psi < 0$ in $\R$. By hypothesis of $v$, we have that $(\Lop_{c,b} + \rho) v \geq \rho v > 0$. As $v$ is $L^1(\R)$ and at least $C^1$, it follows that $v\rightarrow 0$ as $|x|\rightarrow \infty$.

By \eqref{hyp:bad_or} we have that $\frac{f(x,s)}{s} < -\nu$ for $|x| > R_0$ and any $s\in \R$. Fix now $\rho < \min\{\lambda_1(\Lop_{c,b}), \nu\}$, so that $b + \rho < 0$ outside $[R_0, R_0]$. Then, as $v$ is bounded and $\psi >0$, for $C$ large enough we have $v < C\psi$ in $[-R_0,R_0]$. We can then apply Proposition \ref{prop:mp} and the strong maximum principle to $v-C\psi$, which satisfies
\begin{equation}
\label{eq:comparison_lambda}
\begin{cases}
\Delta^s w  +cw'+ (b + \rho) w >  0 & \inSet{\R},\\
w\leq 0 & \text{ inside } [-R_0, R_0], \\
\limsup_{|x|\rightarrow \infty} w \leq 0,
\end{cases}
\end{equation}
to conclude that $v < C\psi$ in $\R \setminus [-R_0,R_0]$. 

We claim that $v< \tau\psi$ for every $\tau \in (0, C]$. Aiming to use the sliding method, define
\[T := \inf \{ t > 0: v < t \psi \text{ in }\R\}.\]
If $T>0$, then we have $v\leq T\psi$ in $\R$. Posing $\Tilde{w} = v-T\psi$, we have that $\Tilde{w}$ satisfies \eqref{eq:comparison_lambda}, so that, by the strong maximum principle, $v < T\psi$ in $\R \setminus [-R_0,R_0]$. Denote $S = \frac{1}{T} < \infty$ which satisfies
\[S = \sup \{s > 0: \psi  > t v \text{ in } \R\}.\]

Set $\tilde{\psi} = \psi - Sv$, which is a strict supersolution to $\Lop_{c,b+\rho} \tilde{\psi}<0$ in $\R$. By applying Proposition \ref{prop:mp} we get $\tilde{\psi} > 0$ outside $[-R_0, R_0]$ and by the weak Harnack inequality \ref{prop:weak_harnack} we get $ \tilde{\psi} > 0$ inside $[-R_0,R_0]$ as well. Thus, there exists $C_0$ such that $v < C_0 \tilde{\psi}$ inside $[-R_0,R_0]$. This implies, as before, that $v < C_0  \tilde{\psi}$ in $\R$, that is
 \[ \psi > \left( S + \frac{1}{C_0}\right)v  \] 
 contradicting the maximality of $S$, and then the minimality of $T$. We conclude $T=0$. 

Thus, we get $0<v\leq 0$, a contradiction. This proves $\lambda_1 \leq 0$.

For the second claim, observe first that $\lambda_1(\Lop_{c,b}) = \lambda_1(\Lop_{-c,b})$, since by taking $\phi_1^c$ and $\phi_1^{-c}$ the corresponding eigenfunctions of $\lambda_1^c$ and $\lambda_1^{-c}$ respectively,
\[\int_{\R} (\Delta^s \phi_1^c + c(\phi_1^c)' + (b + \lambda_1^c)\phi_1^c) \phi_1^{-c} = 0 =\int_{\R} (\Delta^s \phi_1^{-c} + c(\phi_1^{-c})' + (b + \lambda_1^{-c})\phi_1^{-c}) \phi_1^c\]
 and so integration by parts leads to
 \[\int_{\R} (\lambda_1^c - \lambda_1^{-c}) \phi_1^c\phi_1^{-c} =0, \]
 that is, $\lambda_1^{c} = \lambda_1^{-c}$.
 
By the above discussion, we have $\lambda_1(\Lop_{-c,b})\leq 0$. Take $(R_n)_{n\in\N}$ such that $R_n\nearrow\infty$ and define $\lambda_n := \lambda_1(\Lop_{-c,b}, (-R_n, R_n))$. By monotonicity,  $\lambda_n  \searrow \lambda_1(\Lop_{-c,b})\leq 0$, so for some $n\in\N$ large we have $b + \lambda_n \leq -\frac{\nu}{2}$ in $(-R_n, R_n)$ for all $R_n > R_0$.

Now take $\phi_n$ eigenfunctions associated with $\lambda_n$, namely, $\phi_n \in C([-R_n, R_n]) \cap C_{loc}^{\alpha}(-R_n,R_n)$ with
\[\begin{cases}
\Delta^s \phi_n - c \phi_n' + (b + \lambda_n) \phi_n =0 & \inSet{(-R_n,R_n)}\\
\phi_n = 0  & \inSet{\R\setminus(-R_n,R_n)},
\end{cases}\]
normalized to have $\sup_{(-R_0, R_0)} \phi_n = 1$. 

Take $\Phi$ as given by Corollary \ref{cor:supersol}, such that 
\[\Delta^s\Phi + c \Phi' - \frac{\nu}{2}\leq 0 \text{ in }\R\setminus[-\Tilde{R}, \Tilde{R}],\]
and take a smooth mollifier $\eta_{\epsilon}$ so that $\Phi \star \eta_{\epsilon} > 0$ in $[-\Tilde{R}, \Tilde{R}]$. 

For $n$ sufficiently large, $R_n > \Tilde{R}$. This implies that, for $C>0$ large enough
\[\begin{cases}
\Delta^s(C\Phi) - c(C\Phi)' + (b+\lambda_n) C\Phi \leq 0 & \text{ in } (-R_n, -\Tilde{R}) \cup  (\Tilde{R}, R_n) \\
\Delta^s(\phi_n) - c(\phi_n)' + (b+\lambda_n) \phi_n = 0 & \text{ in } (-R_n, -\Tilde{R}) \cup  (\Tilde{R}, R_n),
\end{cases}\] 
and so $\phi_n \leq C\Phi_{\epsilon}$ by the comparison principle.

With this, we may use the local regularity theory to conclude that $\phi_n$ is bounded in $C^{2s + \alpha}$. We then extract a subsequence converging to some $\phi$ in $C^{1,\alpha}_{loc}(\R)$, which satisfies
\[\Delta^s \phi - c\phi' + b\phi = - \lambda \phi \geq 0,~\text{ where } \lambda = \lim_{n\rightarrow \infty} \lambda_n\]
Lastly, as $b\in L^{\infty}$, $\phi' \in L^{\infty}$ so we can bootstrap to get $\Delta^s \phi \in L^{\infty}$.
\end{proof}

\subsection{The influence of the drift}

\begin{prop}
\label{prop:continuity_drift}
Let $\Omega\subset \R$ be a bounded domain. Then, the function $c\mapsto \lambda_1(\Lop_{c,a}, \Omega)$ is continuous.
\end{prop}

\begin{proof}
Let $(c_n)_{n\in\N}$ be a sequence converging to $c$. Denote $\lambda_n := \lambda_1(\Lop_{c_n,a}, \Omega)$. By Theorem \ref{thm:eigfunct_existence}, there exist a sequence $(\phi_n) \subset C^{\beta}( \overline{\Omega}) \cap \Cloc^{2s+\alpha}(\Omega)$ of eigenfunctions associated to $\lambda_n$, normalized to have $\phi_n(x_0) = 1$ for some $x_0 \in \Omega$ fixed.

Let us first prove that $(\lambda_n)_{n\in\N}$ is bounded. By Proposition \ref{prop:eig_properties} (iv), it suffices to provide an upper bound. Fix $\sigma > \|a\|_{\infty}$, so that $a-\sigma < 0$. As $\Omega$ is open, there exist some open set $A$ whose $\epsilon$-enlargement $A_{\epsilon}$ is compactly contained in $\Omega$. Take $h$ a smooth function such that $ 0 \leq h \leq 1$, $h \equiv 1$ in $A$ and $h\equiv 0$ outside $A_{\epsilon}$. For $n\in \N$, let $v_n$ be the solution to
\begin{equation}
\label{eq:h}
\begin{cases}
\Delta^s v + c_nv' + (a - \sigma)v = h & \inSet{\Omega} \\
v = 0 & \text{ in } \R\setminus \Omega.
\end{cases}
\end{equation}
Since $h\geq 0$, we have $v_n \leq 0$ by the maximum principle, and as $h \not \equiv 0$, $v_n < 0$ by the strong maximum principle. Since $A_{\epsilon}$ is compactly supported in $\Omega$, it yields $\sup_{A_{\epsilon}} v_n < 0$, so taking 
\[\rho_n = \frac{-1}{\sup_{A_{\epsilon}} v_n} < \infty\]
we have $h \leq - \rho_n v_n$ in $\Omega$, therefore
\[\Delta^s v_n + c_n v' + (a - \sigma + \rho_n) v_n \leq 0 ~~ \text{ in }\Omega. \]
This implies that $-\Lop_{c_n, a-\sigma+\rho_n}$ does not satisfy the maximum principle and thus $\lambda_n \leq \rho_n - \sigma$.

Recalling that $c_n \rightarrow c$, it holds that $v_n\rightarrow v$, where $v$ is a solution to \eqref{eq:h} with $c_n$ replaced by $c$. Again by the weak and the strong maximum principle, $v < 0$ so that $\rho = (-\max_{A_{\epsilon}} v)^{-1} < \infty$. As $\rho_n \rightarrow \rho$, in particular $\sup_{n\in\N}\rho_n \leq C < \infty$. This implies that $\lambda_n \leq C -\sigma$, that is, $(\lambda_n)_{n\in\N}$ is a bounded sequence.

Now, up to a subsequence, $\lambda_n \rightarrow \lambda^*$. By regularity, $\phi_n$ is bounded in $C^{\beta}(\overline{\Omega})$ and, up to a further subsequence, $\phi_n \rightarrow \phi$ uniformly in $\Omega$. Using the Harnack inequality and the fact that $\phi_n(x_0) = 1$ for every $n$, we have $\phi >0$ in $\Omega$. As viscosity solutions are stable under uniform convergence, $\phi$ is a solution to 
\[\begin{cases}
\Delta^s \phi + c\phi' + (a + \lambda^*)\phi = 0 & \inSet{\Omega}\\
\phi > 0 & \inSet{\Omega}\\
\phi= 0 & \inSet{\R^N\setminus\Omega}.
\end{cases}\] 
This implies that $\lambda^* = \lambda_1(\Lop_{c,a},\Omega)$.

This last argument is valid for any subsequence of $(\lambda_n)$, so that we actually have $\lambda_n \rightarrow \lambda_1(\Lop_{c,a},\Omega)$, as we wanted. 
\end{proof}

\begin{rmk}
The same argument can be adapted to show continuity of $s \mapsto \lambda_1(\Delta^s + c\partial_x + a)$, even as $s\nearrow 1$. For this, it suffices to recall that the regularity estimates from \cite{CL12} are stable as $s\nearrow 1$, so the compactness arguments continue to hold.
\end{rmk}

\begin{prop}
\label{prop:upper_threshold}
Let $b:\R\rightarrow \R$ be such that $\limsup_{|x|\rightarrow \infty} b(x) < 0$ and $\sup_{x\in \R} b(x) > 0$. Then there exists $c^* >0$ such that for any $c \in \R$, $|c| > c^*$, one has $\lambda_1(\Lop_{c,b}) > 0$. 
\end{prop}

\begin{proof}
Let $r\geq 1$. By hypothesis, there exist $\nu, R_0 >0$ such that $b(x) < - \nu$ for $|x| > R_0$. Without loss of generality, we can take $R_0 > r$. Define 
\[\phi(x) = \begin{cases}
|x|^{2s-1} & \text{ if } x < -r \\
p(x) & \text{ if } x  \in (-r, - r + \epsilon )\\
(x+ R)^{-\beta} & \text{ if } x > - r + \epsilon
\end{cases} \]
where $R, r > 0$, $R> r$, and $p$ is a polynomial chosen in order to have $\phi \in C^2(\R)$ and $\phi' < 0$ in $\R$. We will find $c^*, \lambda>0$ such that $(\lambda, \phi)$ is an admissible test pair for $\lambda_1(\Lop_{c^*,b},\R)$. This will suffice for proving the proposition, as the case $c <0$ can be treated by taking $x\mapsto \phi(-x)$ as a test function.

For $x<-r$, one has 
\begin{align*}
\Delta^s\phi(x) &= \PV \int_{-\infty}^{-r} \frac{|y|^{2s-1} - |x|^{2s-1}}{|x-y|^{1+2s}}\dy + \int_{-r}^{\infty} \frac{\phi(y) - \phi(x)}{|x-y|^{1+2s}}\dy \\
&= \Delta^s(|x|^{2s-1}) +\int_{-r}^{-r+\epsilon} \frac{p(y) - |y|^{2s-1}}{|x-y|^{1+2s}}\dy + \int_{-r +\epsilon}^{\infty} \frac{(R+y)^{-\beta} - |y|^{2s-1}}{|x-y|^{1+2s}}\dy \\
&\leq \int_{-r}^{-r+\epsilon} \frac{p(y) - |y|^{2s-1}}{|x-y|^{1+2s}}\dy + \int_{-r+\epsilon}^{r} \frac{(R+y)^{-\beta} - |y|^{2s-1}}{|x-y|^{1+2s}}\dy,
\end{align*}
because $|x|^{2s-1}$ is the fundamental solution for the fractional Laplacian and $(y+R)^{-\beta} < |y|^{2s-1}$ for $y>r$. Denote the right hand side of the last inequality as $T_1(x)$. Thus:
\begin{align*}
\Lop_{c,b} \phi(x) \leq \left( \frac{T_1(x)}{|x|^{2s-1}} + \frac{c(2s-1)}{x} + b(x)\right)\phi(x) ~~\forall x<-r.
\end{align*}
By our choice of parameters, $T_1(x)$ is a positive continuous function in $(-\infty, -r]$. Furthermore, $T_1(x) = O(|x|^{-2s})$ as $x\rightarrow -\infty$, so we can take $R_1 \geq R_0$ such that $\frac{T_1(x)}{\phi(x)} + b \leq -\frac{\nu}{2}$ for $x < -R_1$. This implies 
\begin{equation}
\label{eq:bound_left}
\Lop_{c, b} \phi(x) \leq \left(\frac{c(2s-1)}{x} - \frac{\nu}{2}\right) \phi, \text{ for } x<-R_1.
\end{equation} 

For $x > -r +\epsilon$, we have
\begin{align*}
\Delta^s \phi(x) &= \Delta^s(|\cdot|^{-\beta}) (R+x) + \int_{-\infty}^{-r+\epsilon} \frac{\phi(y) - |y+R|^{-\beta}}{|x-y|^{1+2s}}\dy\\
&= \frac{A(\beta)}{(R+x)^{\beta + 2s}} + T_2(x),
\end{align*}
where $A(\beta)$ is a positive constant depending on $\beta$ and
\[T_2(x) = \int_{-\infty}^{-r} \frac{|y|^{2s-1} -|y+R|^{-\beta} }{|x-y|^{1+2s}}\dy + \int_{-r}^{-r+\epsilon}\frac{p(y) - |y+R|^{-\beta}}{|x-y|^{1+2s}}\dy.\]
Thus,
\[\Lop_{c,b} \phi(x) \leq \left( \frac{A(\beta)}{(R+x)^{2s}} + (x+R)^{\beta} T_2(x) - \frac{c \beta}{x + R} + b(x)\right)\phi(x) ~~\forall x>-r+\epsilon.\]

Note that for $x>0$ and $y<0$ we have $|y| \leq |x-y|$, so  
\begin{align*}
T_2(x) &\leq \int_{-\infty}^{-r} \frac{1}{|x-y|^2} \dy + \|p-|\cdot +R|^{-\beta}\|_{C([-r, -r+\epsilon])}\int_{-r}^{-r+\epsilon}\frac{1}{|x-y|^{1+2s}}\dy \\
&\leq \frac{1}{|x+r|} + \frac{\|p-|\cdot +R|^{-\beta}\|_{C([-r, -r+\epsilon])}}{2s} \left( \frac{1}{|x+r-\epsilon|^{2s}} - \frac{1}{|x+r|^{2s}}\right).
\end{align*}
Choosing $\beta \in (0,1)$, we have that $(x+R)^{\beta} T_2(x) = o(1)$ as $x\rightarrow \infty$, then there exists $R_2 \geq R_0$ such that $(x+R)^{-2s}A(\beta) + (x+R)^{\beta} T_2(x) \leq \frac{\nu}{2}$. Plugging this in the last inequality,
\begin{equation}
\label{eq:bound_right}
\Lop_{c,b} \phi(x) \leq \left( - \frac{c \beta}{x + R} -\frac{\nu}{2} \right)\phi(x),~~ \forall x > R_2.
\end{equation}

Finally, for $x\in [-R_1, R_2]$, recall that $\phi >0$, $\phi' <0$ and that $\Delta^s\phi\in L^{\infty}(-R_1, R_2)$ because $\phi\in C^2([-R_1, R_2])$. Denote 
\[\begin{array}{cc}
m:= \inf_{[-R_1, R_2]} \phi > 0 & d := \inf_{[-R_1, R_2]} -\phi' >0,
\end{array}\]
then, for $x\in[-R_1,R_2]$,
\begin{equation}
\label{eq:bound_center}
\Lop_{c,b}\phi(x) \leq \left(\frac{\Delta^s \phi(x)}{\phi(x)} + c\frac{\phi'(x)}{\phi(x)} + b(x) \right) \phi(x) \leq \left(\frac{\|\Delta^s\phi\|_{L^{\infty}(-R_1, R_2)}}{m} - c\frac{d}{m} + \sup_{\R} b\right) \phi(x).
\end{equation}

Defining
\[c^*:= \frac{m}{d}\left(\frac{\|\Delta^s\phi\|_{L^{\infty}(-R_1, R_2)}}{m} + \sup_{\R} b\right)>0\]
we see that for every $c>c^*$, by \eqref{eq:bound_left}, \eqref{eq:bound_right} and \eqref{eq:bound_center}, $\Lop_{c,b} \phi(x) \leq -\lambda \phi(x)$ for some $\lambda >0$. We conclude that $(\lambda, \phi)$ is an admissible test pair for $\lambda_1$, so $\lambda_1(\Lop_{c,b})>0$ for every $c>c^*$. 
\end{proof}

\begin{prop}
\label{prop:lower_threshold}
Let $a \in C(\R)$ be such that $\lambda_1(\Lop_{0,a}) <0$. Then there exists $c^{**}>0$ such that $\lambda_1(\Lop_{c,a}) < 0 $ for any $|c| < c^{**}$. 
\end{prop}

\begin{proof}
By the continuity of the principal eigenvalue with respect to the domain, there exists $R>0$ such that $\lambda_1(\Lop_{0,a}, (-R,R)) < \frac{\lambda_1(\Lop_{0,a})}2$. Use Proposition \ref{prop:continuity_drift} to obtain the existence of $c^{**} > 0$ such that $\lambda_1(\Lop_{c,a}, (-R,R)) \leq \frac{\lambda_1(\Lop_{0,a})}{4}$ for any $|c| < c^{**}$. Using that the principal eigenvalue is decreasing with respect to the domain, we have that $\lambda_1(\Lop_{c,a}) \leq \frac{\lambda_1(\Lop_{0,a})}{4} < 0$ for any $|c|\leq c^{**}$. 
\end{proof}

\begin{rmk}
Theorem \ref{thm:thresholds} follows by combining Propositions \ref{prop:upper_threshold} and \ref{prop:upper_threshold}.
\end{rmk}
\section{Existence and uniqueness}
\label{sect:existence}

In this section we complete the proof of Theorem \ref{thm:tw}. The proof of existence is classical and that of nonexistence follows closely the one by Coville in \cite{C21}, with some simplification due to the fact that we have an eigenfunction in the whole of $\R$.

\subsection{Existence}

\begin{prop}
\label{prop:existence}
Assume $\lambda_1 < 0$. Then there exists a bounded, positive solution to \eqref{eq:tw}.
\end{prop}

\begin{proof}
It suffices to prove that there exist a supersolution and a subsolution. By the assumptions on $f$, we know that $f(x,M) < 0$ for all $x\in \R$, if $M$ is a sufficiently large constant. Then, taking $\psi \equiv M$ we have a supersolution to \eqref{eq:tw}.

On the other hand, $\exists R_0 >0$ such that $\lambda_1(R_0) < 0$. We consider $\phi$ the eigenfunction associated to $(-R_0, R_0)$, i.e. $\phi$ satisfies
\[\begin{cases}
\Delta^s \phi + c\phi' + (a + \lambda_1(R_0))\phi = 0 & \inSet{(-R_0,R_0)} \\
\phi = 0 & \inSet{\R\setminus(-R_0,R_0)}.
\end{cases}\]
Take $\epsilon > 0$. For  $|x| \geq R_0$, we have 
\begin{align*}
\Delta^s [\epsilon\phi](x) + c(\epsilon\phi)'(x) + f(x,\epsilon\phi(x)) = \int_{\R} \frac{\epsilon\phi(y)}{|x-y|^{1+2s}} \dy \geq 0, 
\end{align*}
and for $|x|< R_0$, we get
\begin{align*}
\Delta^s [\epsilon\phi](x) + c(\epsilon\phi)'(x) + f(x,\epsilon\phi(x)) &= -(a(x) + \lambda_1(R_0))\epsilon\phi + f(x,\epsilon\phi)\\
&= -\lambda_1(R_0) \epsilon\phi(x) + \left( \frac{f(x,\epsilon\phi)}{\epsilon\phi} - \partial_s f(x,0)\right)\epsilon \phi(x) \\
&= -\lambda_1(R_0) \epsilon\phi(x) + o(\epsilon \phi(x)),
\end{align*}
so taking $\epsilon>0$ sufficiently small, we obtain
\[\Delta^s [\epsilon\phi](x) + c(\epsilon\phi)'(x) + f(x,\epsilon\phi(x)) \geq 0 ~\inSet{\R},\]
that is, $\epsilon \phi$ is a subsolution.
\end{proof}

\begin{prop}
Assume $\lambda_1 \geq 0$. Then there is no nontrivial solution to \eqref{eq:tw}.
\end{prop}

\begin{proof}
By contradiction, suppose $v$ is a solution to \eqref{eq:tw}. By Proposition \ref{prop:decay}, we have $v\in L^1(\R)$ so Proposition \ref{prop:subsol} implies $\lambda_1(\Lop_{c,b})\leq 0$. Note that $b = \frac{f(x,u)}{u} < \partial_s f(x,0) = a$, so by monotonicity of the principal eigenvalue we have $\lambda_1 \leq \lambda_1(\Lop_{c,b})\leq 0$, thus $\lambda_1 = 0$.

Take $R_0 >0$ and $\chi \in C_c^{\infty}(\R)$ such that $0\leq \chi \leq 1$, $\chi \equiv 1$ in $(-\frac{R_0}{2},\frac{R_0}{2})$ and $\mathrm{supp}(\chi) \subset (-R_0, R_0)$. As $b < a$, there exists some $\epsilon >0$ such that $b< b+ \epsilon \chi <a$, so $\lambda_1(\Lop_{c,b+\epsilon\chi})=0$. 

By Theorem \ref{thm:eigfunct_existence}, we can take $\psi \in C^{2s + \alpha}(\R)$ such that 
\[ \Delta^s \psi + c\psi' + (b + \epsilon\chi) \psi = 0\]
which exists by Proposition \ref{prop:eigfunct_continuity}. By repeating the argument of Propositions \ref{prop:decay} and \ref{prop:H1}, $\psi \in W^{1,\infty}(\R) \cap H^1(\R)$. Let $\phi$ be the function given by Proposition \ref{prop:subsol}, namely, 
\[\Delta^s \phi - c\phi' +b \phi \geq 0.\]
Multiply this equation by $\psi$ and integrate over $\R$ to get
\begin{equation}
\label{eq:setup_contradiction}
\int_{\R} \psi \Delta^s \phi - c \psi \phi' + b\psi\phi \geq 0.
\end{equation}
An integration by parts yields
\[\int_{\R} \psi \Delta^s \phi - c \psi \phi' + b\psi\phi = \int_{\R} \phi(\Delta^s \psi + c\psi' + b \psi) = -\epsilon \int_{\R} \chi \psi \phi 
< 0\]
a contradiction with \eqref{eq:setup_contradiction}. Therefore, there is no solution to \eqref{eq:tw} if $\lambda_1 \geq 0$.
\end{proof}

\subsection{Uniqueness}

\begin{prop}
There is at most one bounded positive solution to \eqref{eq:tw}.
\end{prop}

\begin{proof}
By contradiction, let $u_1, u_2$ be two bounded positive solutions to \eqref{eq:tw}. By the estimates of section \ref{sect:a_priori}, we know $u_1,u_2 \in W^{1,\infty}(\R)\cap L^{1}(\R) \cap H^1(\R)$. Define $v = \max\{u_1,u_2\}$, we have then that $ v \in W^{1,\infty}(\R)\cap L^{1}(\R) \cap H^1(\R)$ is a subsolution to \eqref{eq:tw}. This, by Proposition \ref{prop:subsol}, implies the existence of $\phi$ a classical solution to
\begin{equation}
\label{eq:subsol2}
\Delta^s \phi - c \phi' + \frac{f(x,v)}{v} \phi \geq 0~ \text{ in }\R.
\end{equation}

Multiply \eqref{eq:subsol2} by $u_1$ and integrate over $\R$ to get
\[\int_{\R} u_1 \Delta^s \phi - cu_1\phi' + \frac{f(x,v)}{v} u_1 \phi \geq 0.\]
By the regularity and integrability properties of $u_1$ and $\phi$, we can integrate by parts to obtain
\[\int_{\R} \phi \Delta^su_1  + c\phi u_1' + \frac{f(x,v)}{v} u_1 \phi \geq 0,\]
which, using the equation satisfied by $u_1$, implies
\[\int_{\R} \left(\frac{f(x,v)}{v} - \frac{f(x,u_1)}{u_1}\right)u_1 \phi \geq 0.\]

Recall that $s \mapsto \frac{f(x,s)}{s}$ is strictly decreasing, so if $v(x_0)\neq u_1(x_0)$ for some $x_0$, by continuity we must have $v\neq u_1$ in a set of positive measure. The fact that $u_1, \phi >0$ implies
\[\int_{\R} \left(\frac{f(x,v)}{v} - \frac{f(x,u_1)}{u_1}\right)u_1 \phi < 0,\]
a contradiction, and the proposition follows.
\end{proof}

\section{Long time behavior}
\label{sect:parabolic}
In what follows, 
\[\Cub(\R) = \{u:\R \rightarrow \R: u \text{ is uniformly continuous and bounded}\}.\]
By the general existence and uniqueness theory for Feller semigroups (see \cite{CR13}, Section 2.3) and hypotheses \eqref{hyp:KPP} and \eqref{hyp:lipschitz}, for any $u_0 \in \Cub$ there exists a unique solution $u(t,x)$ to \eqref{eq:parabolic}.

Note that any solution $v$ to \eqref{eq:parabolic} gives way to a solution $v_c$ to \eqref{eq:parabolic_c} by posing $v_c(x,t) = v(x + ct, t)$.
\begin{equation}
\label{eq:parabolic_c}
\begin{cases}
\partial_t u = \Delta^s u + cu' + f(x,u) & (x,t) \in \R\times (0,\infty), \\
u(0,\cdot) = u_0 & x\in \R.
\end{cases}
\end{equation}
We will deduce the behavior of equation \eqref{eq:parabolic} by studying suitable solutions to \eqref{eq:parabolic_c}. 

We begin by stating a comparison principle for the nonlinear equation. Its proof is rather standard, we provide it here for convenience of the reader.
\begin{prop}
\label{prop:comparison}
Let $u,v \in C^1([0,\infty), C_{ub}(\R))$, $c\in \R$, and let $f:\R\times \R \rightarrow \R$ be continuous in the first variable and globally Lipschitz with respect to the second variable, uniformly with respect to the first. Then
\[\begin{cases}
\partial_t v\geq \Delta^s v + cv_{x} + f(x,v) & \text{ in } (0,\infty)\times \R\\
\partial_t u \leq \Delta^s u + cu_x + f(x,u)  & \text{ in } (0,\infty)\times \R\\
u(0, \cdot) \leq v(0,\cdot) & \text{ in } \R \\
\displaystyle\limsup_{|x|\rightarrow\infty} u (t,x) \leq \displaystyle\liminf_{|x|\rightarrow \infty} v(t,x)& \text{ uniformly in }t\in[0, T],~~ \forall T> 0
\end{cases}\] 
implies $u\leq v$ in $[0,\infty)\times\R$.
\end{prop}
\begin{proof}
Note that the operator $A = (-\Delta)^s - c\partial_x$ generates a strongly continuous Feller semigroup as a consequence of the heat kernel estimates proven in \cite{BJ07}.

Denote $M$ as the Lipschitz constant of $f$. For $\lambda > 0$, define $w(t,x) = e^{-\lambda t} (u-v)$. It is easy to see that $w$ satisfies
\begin{align*}
\partial_t w + (-\Delta)^s w - c\partial_x w &\leq e^{-\lambda t} (f(x,u) - f(x,v)) - \lambda w \\
&\leq Mw - \lambda w
\end{align*}
for all $(t,x) \in (0,\infty) \times \R$. Taking $\lambda > M$, we have that $\partial_t w +(-\Delta)^sw - cw' \leq 0$. Furthermore, $w(0, \cdot) \leq 0$ and \[\limsup_{|x|\rightarrow\infty} w(t,x) \leq \limsup_{|x|\rightarrow\infty} u(t,x) - \liminf_{|x|\rightarrow\infty} v(t,x) \leq 0\]
uniformly in $t\in [0, T]$ for any $T>0$. We can then apply the maximum principle from \cite{CR13}, Proposition 2.8, to conclude that $w\leq 0$ in $[0,\infty)$, i.e. $u\leq v$.
\end{proof}

\begin{rmk}
\label{rmk:parabolic_smp}
The proof of Proposition \ref{prop:comparison} also shows that if $v(t,x)$ is a solution to \eqref{eq:parabolic_c}, then it cannot have a minimum in $(0,\infty)\times \R$. In particular, if $u_0$ is nonnegative and nontrivial, then $u_c(t) > 0$ for any $t>0$. Indeed, by the same argument as above, it suffices to prove it for $v$ solution to $v_t \geq \Delta^s v +c \partial_x v$. If $(t_0,x_0)$ is a negative minimum of $v$, then $\partial_t v(t_0, x_0) = \partial_x v(x_0,t_0) = 0$ and $\Delta^s v(t_0, x_0) > 0$, and one concludes as in Proposition \ref{prop:mp}.  
\end{rmk}

The following corollary is key in proving the long time behavior of the parabolic problem. 

\begin{cor}
\label{cor:monotonicity}
Let $u_c(x,t)$ be a solution to \eqref{eq:parabolic_c} with nonnegative initial condition $u_0 \in C_{ub}(\R)$. If $u_0$ is a supersolution to \eqref{eq:tw}, then $u$ is nonincreasing in time. Conversely, if $u_0$ is a subsolution to \eqref{eq:tw}, $u$ is nondecreasing in time.
\end{cor}
\begin{proof}
It suffices to show the first statement, as the second is analogous to the first. Take $v(x,t) = u_0(x)$, so $\partial_t v = 0$ in $(0,\infty) \times \R$ and thus $\partial_t v \geq \Delta ^s v +c\partial_x v + f(x,v)$ since $u_0$ is a supersolution to \eqref{eq:tw}. Let $T>0$, and we claim that 
\begin{equation}
\label{eq:boundary_parabolic}
\limsup_{|x|\rightarrow\infty} u (t,x) \leq 0 \leq \liminf_{|x|\rightarrow \infty} u_0(x)
\end{equation}
uniformly in $t\in [0,T]$ for any $T>0$. Indeed, by Theorem 2 in \cite{BJ07}, there exist $C>0$ such that the heat kernel $p(t,x)$ of $A = (-\Delta)^s - c\partial_x$ is bounded by
\[\frac{C^{-1}t}{t^{\frac{1+2s}{2s}} + |x|^{1+2s}} \leq p(t,x) \leq \frac{Ct}{t^{\frac{1+2s}{2s}} + |x|^{1+2s}}\] 
for all $x \in \R$, $t\in [0,T]$. 

By the Duhamel formula, we know that $u$ satisfies
\begin{equation}
\label{eq:duhamel}
u(x,t) = P_t u_0(x) + \int_0^t P_{t-s} f(\cdot, u(s,\cdot))\ds
\end{equation}
where $P_t$ is the semigroup generated by $A$ and has the representation formula
\[P_tw(x) = \int_{\R} p(t,x-y) w(y) \dy\]
and is such that $\|P_t w - w  \|_{\infty} \rightarrow 0$ as $t\searrow 0$. Take $t_0>0$ such that $\| P_t w - w \|_{\infty} \leq \frac{\delta}{4}$ for $t\in [0,t_0]$. Therefore, \eqref{eq:boundary_parabolic} will hold for $t\in [0,t_0]$.  

For $t\in [t_0, T]$, we can bound each of the terms in \eqref{eq:duhamel} as follows. For the homogeneous part $P_t u_0(x)$ we have
\begin{align*}
P_tu_0(x) &\leq C\int_{\R}\frac{tu_0(y)}{t^{\frac{1+2s}{2s}} + |x-y|^{1+2s}}\dy \\
&\leq C\|u_0\|_{\infty}   \int_{\R} \frac{T}{t_0^{\frac{1+2s}{2s}} + |x-y|^{1+2s}}\dy
\end{align*}
for any $L \geq R$. This last integral is finite, so we can choose $R_1$ such that for $|x| > R_1$
\[ \int_{\R} \frac{T}{t_0^{\frac{1+2s}{2s}} + |x-y|^{1+2s}}\dy \leq \frac{\delta}{C\|u_0\|_{\infty}}.\] 
Thus, for $t\in [t_0, T]$ and $|x| >R_1$, we have
\[P_t u_0(x) \leq  \delta.\]

For the inhomogeneous part, recall that for $|y|>R_0$, $f(y,s) \leq -\nu s$, so 
\begin{align*}
\int_0^t P_{t-s} f(\cdot, u(s,\cdot))\ds &= \int_0^t \left( \int_{-R_0}^{R_0}p(t-s, x-y) f(y, u(s,y)) \dy + \int_{(-R_0, R_0)^c} p(t-s, x-y) f(y, u(s,y)) \dy\right)\ds\\
&\leq C \int_0^t \int_{-R_0}^{R_0}\frac{(t-s)f(y, u(s,y))}{(t-s)^{\frac{1+2s}{2s}} + |x-y|^{1+2s}} \dy \ds.
\end{align*}
Since $f$ is bounded on bounded sets, we can choose $R_2$ such that for $|x|>R_2\geq R_0+1$, this integral is smaller than $\delta$. This will imply that 
\[u(t,x) \leq 2\delta\]
and the claim \eqref{eq:boundary_parabolic} follows. Thus, we can apply the preceding proposition to conclude $u(t,x) \leq v(t,x)$ for all $(t,x) \in (0,\infty)\times \R$, and Remark \ref{rmk:parabolic_smp} implies further that $u(t,x) <u_0(x)$ in $(0,\infty) \times \R$.

Now, fix $t>0$ and define
\[T := \sup\{s\geq 0: u_c(t,x) < u_c(\tau,x)~ \forall x\in \R, \tau\in [0, s) \}.\]
Note that $0 < T\leq t$. By contradiction, if $T< t$, then there exist $0< s_0 < t$ and $x_0\in \R$ such that $u_c(t, x_0) = u_c(s_0, x_0)$ and $u_c(t,x) < u_c(s,x)$ for all $s< s_0$. Define  $w(s,x) = v(s,x) - v(t,x)$, and thus $w_t = \Delta^s w + c\partial_x w + \left(\frac{f(x,u(s,x))- f(x,u(t,x))}{u(s,x)- u(t,x))}\right)w$, with $w$ having a minimum at $(s_0,x_0) \in (0,\infty) \times \R$. By Remark \ref{rmk:parabolic_smp}, this cannot be, so we must have $T=t$. This concludes the proof. 
\end{proof}

We have the necessary tools to examine the behavior of solutions to \eqref{eq:parabolic}. 
\begin{proof}[Proof of Theorem \ref{thm:convergence}]
This proof follows the same lines as \cite{B+09}, Theorem 4.11, and \cite{C21}, Section 8.

As a first step, let us prove that there is local uniform convergence. Note that if $u$ is the solution to \eqref{eq:parabolic}, then $u_c(t,x) = u_c(t, x+ct)$ is the solution to \eqref{eq:parabolic_c} with the same initial condition. 

Take $z_c(t,x)$ as the solution to $\partial_t z = \Delta^s z + c \partial_x z +f(x,z)$ with initial condition $z_0 = C\|u_0\|_{\infty}$, where $C>1$ is chosen so large as to have $f(x,z_0(x))\leq 0$ for all $x$. Recall that $z_c$ exists by the fact that $A = (-\Delta)^s - c\partial_x$ generates a strongly continuous Feller semigroup. By Corollary \ref{cor:monotonicity}, we see that $z_c$ is nonincreasing in time. As $z_c$ is nonnegative, we must have that $z_c(t,x)\searrow \overline{z}(x)$ as $t\rightarrow \infty$, where $\overline{z}$ is a nonnegative solution to \eqref{eq:tw}. Arguing as in the proof of Corollary \ref{cor:monotonicity}, $z_0 \geq u_0$ implies that $z_c \geq u_c \geq 0$, so that 
\begin{equation}
\label{eq:limsup}
\limsup_{t\rightarrow \infty} u_c(t,x) \leq \lim_{t\rightarrow\infty} z_c(t,x).
\end{equation}

If $\lambda_1 \geq 0$, the only nonnegative solution is the trivial one, so we must have that $z_c(t,x) \rightarrow 0$ as $t\rightarrow \infty$ for all $x\in \R$, and by \eqref{eq:limsup} we must have $u_c(t,x) \rightarrow 0$, and thus $u(x,t) \rightarrow 0$, for all $x\in\R$.

If $\lambda_1 < 0$, it remains to show that $\overline{z}$ is not identically zero. By Remark \ref{rmk:parabolic_smp}, if $u_0$ is nontrivial, then $u(1,x)> 0$ and therefore $u_c(1,x)>0$. For any $R> 0$ there exists $\epsilon >0$ such that $u(1,x)>\epsilon\phi_R$ in $\R$, where $\phi_R$ is the principal eigenfunction of $\Lop_{c,a}$ in $(-R,R)$. As discussed in the proof of Proposition \ref{prop:existence}, we can choose $R>0$ so that $\epsilon\phi_R$ is a subsolution to \eqref{eq:tw} in $(-R, R)$, and outside it is easy to see that 
\[\Delta^s \phi_R(x) + c \phi_R'(x) + f(x,\phi_R(x)) = \int_{-R}^R \frac{\phi(y)}{|x-y|^{1+2s}}\dy >0 \]
for $|x| >R$. Then, $\phi_R$ is a subsolution to \eqref{eq:tw} in $\R$. 

Take $h_c(t,x)$ as the solution to \eqref{eq:parabolic_c} with initial condition $\epsilon \phi_R$, so that $h_c(t,x) \leq u(t+1,x)$ and $h_c(\cdot,x)$ is increasing. We get that $\overline{h}(x) = \lim_{t\rightarrow \infty} h_c(t,x)$ exists and satisfies
\begin{equation}
\label{eq:liminf}
\overline{h}(x) \leq \liminf_{t\rightarrow \infty} u(x,t+1) \leq \limsup_{t\rightarrow \infty} u_c(t,x) \leq \overline{z}(x).
\end{equation}
Moreover, $\overline{h}$ is a solution to \eqref{eq:tw} which is nontrivial by monotonicity of $h_c$. By uniqueness and \eqref{eq:liminf}, we obtain that $\overline{h} = \overline{z} = u_{\infty}$ is the unique nontrivial solution to \eqref{eq:tw} and therefore $\lim_{t\rightarrow \infty} u_c(x,t)$ exists and coincides with $u_{\infty}$.

Let us now improve the convergence to make it uniform. By contradiction, suppose there exist $(t_n), ~(x_n)\subset \R$ and $\delta >0$ such that $t_n \rightarrow \infty$ and that 
\begin{equation}
\label{eq:anti_unif}
z(t_n,x_n) - u_{\infty}(x_n) \geq \delta.
\end{equation}

Define now 
\[z_n(t,x) := z_c(t,x + x_n).\]
Suppose $(x_n)$ is bounded. Then, up to a subsequence, we may assume $x_n\rightarrow x_{\infty}$. By parabolic regularity estimates (from \cite{CLD12}, Theorem 5.2), $z_n$ converges locally uniformly to some $z_{\infty}$, where $z_{\infty}$ is a solution to 
\[\partial_t  z_{\infty}(t,x) = \Delta^s z_{\infty} (t,x) +c\partial_x z_{\infty}(t,x) + f(x + x_{\infty}),\]
which in turn converges, as $t\rightarrow \infty$, to a solution $\bar{z}_{\infty}$ of 
\[\Delta^s z_{\infty} (t,x) +c\partial_x z_{\infty}(t,x) + f(x + x_{\infty}, z_{\infty}(t,x)) = 0.\]
By uniqueness, one must have $\bar{z}_{\infty}(x) = u_{\infty}(x + x_{\infty})$, in particular $\bar{z}_{\infty}(0) = u_{\infty}(x_{\infty})$. Rewriting \eqref{eq:anti_unif} in terms of $z_n$, we have 
\[z_n(t_n, 0) - u_{\infty}(x_n) \geq \delta.\]
Taking the limit $n\rightarrow \infty$ in this last inequality we arrive at a contradiction.

Thus, $(x_n)$ must be unbounded, say $|x_n| \rightarrow \infty$. For $n$ sufficiently large, there is $\nu > 0$ such that $z_n$ satisfies
\[\partial_t z_n \leq \Delta^s z_n + c z_n - \nu z_n.\]
Recall that $0 < z_n(t,x) <C\|u_0\|_{\infty}$, and again using parabolic regularity estimates we find that 
\[\|z_n(t,\cdot)\|_{C^{2s+\alpha}} \leq C(\|u_0\| + \|f\|_{L^{\infty}(\R\times [0, C\|u_0\|]}),\] 
so that, up to a subsequence, one has $z_n\rightarrow z_{\infty}$, which will converge as $t\rightarrow \infty$ to a solution of 
\[\Delta^s  \bar{z}_{\infty} + c\bar{z}_{\infty} - \nu \bar{z}_{\infty} \geq 0 ~~ \inSet{\R}.\]
Arguing as in the proof of Proposition \ref{prop:mp} it easy to see that $\bar{z}_{\infty}$ cannot have a positive maximum, so we obtain that $\bar{z}_{\infty} \equiv 0$. But this cannot hold, because \eqref{eq:anti_unif} implies, for $n$ sufficiently large, that
\[0\leq u_{\infty}(x_n) \leq -\frac{\delta}{2} < 0.\]

We conclude that the convergence must be uniform in space. 
\end{proof}
{\small

	{\bf Aknowlegments}. S.\ Flores-Sepúlveda and G.\ Nornberg were supported by Centro de Modelamiento Matemático (CMM) BASAL fund FB210005 for center of excellence from ANID-Chile; and by ANID Fondecyt grant 1220776.
	
	A. Quaas was supported by FONDECYT Grant 1231585.

	{\bf Data Availability}: Data sharing is not applicable to this article because we do not analyse or generate any datasets.
}

\end{document}